\documentclass{article}
\usepackage[utf8]{inputenc}
\usepackage[a4paper, total={6in, 8in}]{geometry}
\usepackage{hyperref}
\hypersetup{
    colorlinks=true,
    linkcolor=blue,
    filecolor=magenta,      
    urlcolor=cyan,
}
\usepackage{url}
\usepackage{graphicx}
\usepackage{amssymb}
\usepackage{amsmath}
\usepackage{amsthm}
\usepackage{subfig}
\usepackage{float}
\usepackage[ruled,vlined,linesnumbered]{algorithm2e}
\usepackage{siunitx,booktabs,threeparttable,caption}
\theoremstyle{definition}
\newtheorem{definition}{Definition}
\theoremstyle{plain}
\newtheorem{theorem}{Theorem}
\newtheorem{lemma}{Lemma}
\usepackage{dirtytalk}
\usepackage{float}
\usepackage{multicol}
\usepackage[justification=centering]{caption}
\usepackage[title]{appendix}
\usepackage{nomencl}
\makenomenclature

\title{ITSO: A novel Inverse Transform Sampling-based Optimization algorithm for stochastic search}
\author{\normalsize Nikolaos P. Bakas\thanks{Computation-based Science and Technology Research Center, The Cyprus Institute, 20 Konstantinou Kavafi Street, 2121, Aglantzia Nicosia, Cyprus. e-mail: n.bakas@cyi.ac.cy}, Vagelis Plevris\thanks{Department of Civil Engineering and Energy Technology, OsloMet—Oslo Metropolitan University, Pilestredet 35, Oslo 0166, Norway. e-mail: vageli@oslomet.no}, Andreas Langousis\thanks{Department of Civil Engineering, University of Patras, 265 04 Patras, Greece. e-mail: andlag@alum.mit.edu}, Savvas A. Chatzichristofis\thanks{Intelligent Systems Lab \&  Department of Computer Science, Neapolis University Pafos, 2 Danais Avenue, 8042 Pafos, Cyprus. e-mail: s.chatzichristofis@nup.ac.cy} }
\date{}

\begin{document}

\maketitle

\begin{abstract}
Optimization algorithms appear in the core calculations of numerous Artificial Intelligence (AI) and Machine Learning methods, as well as Engineering and Business applications. Following recent works on the theoretical deficiencies of AI, a rigor context for the optimization problem of a \textit{black-box} objective function is developed. The algorithm stems directly from the theory of probability, instead of a presumed inspiration, thus the convergence properties of the proposed methodology are inherently stable. In particular, the proposed optimizer utilizes an algorithmic implementation of the $n$-dimensional inverse transform sampling as a search strategy. No control parameters are required to be tuned, and the trade-off among exploration and exploitation is by definition satisfied. A theoretical proof is provided, concluding that only falling into the proposed framework, either directly or incidentally, any optimization algorithm converges in the fastest possible time. The numerical experiments, verify the theoretical results on the efficacy of the algorithm apropos reaching the optimum, as fast as possible.

\vspace{5mm}
Keywords: Stochastic Optimization, Inverse Transform Sampling, Black-box Function, Global Convergence.
\end{abstract}

\section{Introduction}
Despite the numerous research works and industrial applications of Artificial Intelligence algorithms, they have been criticized about lacking a solid theoretical background \cite{Hutson2018}. The empirical results demonstrate impressive performance, however, their theoretical foundation and analysis are often vague \cite{Sculley2018}. Machine Learning (ML) models, frequently aim at identifying an optimal solution \cite{sra2012optimization,bottou2018optimization}, which is computationally hard and attempts to explain the procedure often based on the evaluation of the function's Gradients \cite{Rudin2019,wu2017bayesian}. Accordingly, the identification of whether a stochastic algorithm will work or not remains an open question for many research and real-world applications \cite{doerr2020probabilistic,doerr2020complexity,papadrakakis2005design,papadrakakis2001optimum}. A vast number of optimization algorithms have been developed and applied for the solution of the corresponding problems \cite{moayyeri2019cost,plevris2011hybrid,Lagaros2009,LAGAROS2006}, as well as mathematical proofs regarding their algorithmic convergence \cite{bull2011convergence,rudolph1994convergence,clerc2002particle}, however, their basic formulation often stems from nature-inspired procedures \cite{opara2019differential,razmjooy2019comprehensive} and not solid mathematical frameworks. Accordingly, a vast number of research works have been published in order to investigate the performance of black-box algorithms \cite{munoz2017performance,audet2017derivative,audet2016blackbox}. 

In its elementary form, the purpose of efficient optimization algorithms is to find the argument yielding the minimum value of a \textit{black-box} function $f(x)$, defined on a set $A$, $f:A \to \mathbb {R}^n$. Accordingly, the inverse problem of maximization is the minimization of the negated function $-f(x)$, while problems with multiple objective functions often utilize single function optimization algorithms to attain the best possible solution. $A$ is assumed a compact subset of the Euclidean space ${{\mathbb{R}}^{n}}$, where $n$ is the number of dimensions of the set $A$, however, the proposed method applies similarly to discrete and continuous topological spaces. The unknown, black-box function $f$, returns values for the given input $x_{ij}={{\mathbf{x}}_{i}}=({{x}_{i1}},{{x}_{i2}},\cdots ,{{x}_{in}})$ at each computational discrete time step $i=\left\{ 1,2,\ldots ,{{f}_{e}} \right\}$, where ${{f}_{e}}$ is the number of maximum function evaluations. %and $j=\left\{ 1,2,\ldots ,n \right\}$ dimension of $\mathbf x_i$. 
The sought solution is a vector ${{\mathbf{x}}_{min}}\in A$, such that $f({{\mathbf{x}}_{min}})\le f(\mathbf{x}),\forall \mathbf{x}\in A$, which may be written by 
\begin{equation}
\begin{split}
{{\mathbf{x}}_{min}}={\operatorname {arg\,min} }\,f(\mathbf{x}):=\{\mathbf{x}\in A \subseteq \mathbb{R}^n 
\\
\mid \forall \mathbf{y}\in A:f(\mathbf{y})\ge f(\mathbf{x})\}
\end{split}
\end{equation}

The purpose of this work is to provide a rigor context for the optimization problem of a \textit{black-box} function, by adhering to the Probability Theory, aiming at identifying the best possible solution $\mathbf{x}_{min}$, within the given iterations $f_e$, during the execution of the algorithm. 

%Section \ref{sec:fast} presents the fast formulation of the proposed algorithm, which may rapidly be implemented, and offers an initial introduction to the algorithm. It is worth mentioning that the simple version of the algorithm which is described in one line of algorithmic code attains high accuracy for cost demanding problems, making it useful due to its vastly fast implementation, as well as for problems with time-consuming unknown functions.

The rest of the paper is organized as follows:
In Section \ref{sec:iit}, the proposed Inverse Transform Sampling Optimizer (ITSO) is presented. Additionally, the same Section provides in details  the theoretical formulation of the algorithm as well as some programming and implementation techniques. Illustrative examples of the optimization history are also comprised. In Section \ref{sec:proof} the theoretical proof of convergence is provided, as well as Lemma \ref{lemma-conv}, deriving that the suggested optimization framework is the fastest possible. The numerical experiments are divided into three groups. Subsection \ref{sec:bbf} is about the comparison with 13 nonlinear loss functions, 17 optimization methods, for 10 and 20 dimensions of search space, and 5000 and 10000 iterations per dimension. Section \ref{sec:prog}, briefly presents the  programming techniques that were investigated, in order to implement the proposed method into a computer code. Finally, the conclusions are drawn in Section \ref{sec:conclusions}.

%Subsection \ref{sec:reinf} presents the application of the algorithm for Reinforcement Learning, and Subsection \ref{sec:bbf} for Structural Optimization. 

\section{Optimization by Inverse Transform Sampling}
\label{sec:iit}

Let the probability distribution of the optimal values of $f$, be considered as the product of some monotonically decreasing kernel $k$ over $f$ at some time-step $i$ and hence given input ${{\mathbf{x}}_{i}}$. This assumption stands in the foundation of the method and can be considered as rational, instead of an arbitrary selection strategy, inspired by natural or other phenomena. It is a straightforward application of the Probability theory to the problem. The kernel may be considered as parametric, concerning time $i$. The selection of the kernel $k$ should satisfy the condition that its limit is the Dirac delta function centered at ${\operatorname {arg\,min} }\,f$,
\begin{equation}
    {\mathop{\lim_{i\to \infty }}}\,k_i(f)={{\delta }_{m}}(x),
    \label{eq:dirac}
\end{equation}
where $\delta_{m}$ denotes the Dirac function centered at ${\operatorname {arg\,min} }\,f$. Although the selection of the kernel $k$ is ambiguous, it can be chosen among a variety of functions satisfying Equation \ref{eq:dirac}, such as the Gaussian:
\begin{equation}
k_i(f)=exp(-{{f(\mathbf x_i)}^{2}}g(i)),
\end{equation}
where $g(i)$ is a time increasing pattern. The function $g(i)$ controls the shape of the kernel $k$, approximating numerically Equation \ref{eq:dirac}. Additionally, we may use:
\begin{equation}
k_i({f}) = {\operatorname {max} }\,f-f(\mathbf x_i),
\end{equation}
\begin{equation}
k_i({f}) = 1-\frac{f(\mathbf x_i)-{\operatorname {min} }\,f+e_i}{{\operatorname {max} }\,f-{\operatorname {min} }\,f+e_i} \wedge {e_i\to 0},
\end{equation}
or any other non-negative Lebesgue-integrable function, which reverses the order of the given set of all ${{f}_{{\hat{i}}}}$, where $\hat{i}$ is the permutation of the indices $1,2,...,i$, such that the sequence of ${{f}_{{\hat{i}}}}$ being monotonically strictly decreasing. The duplicate values of ${{f}_{{\hat{i}}}}$ should be extracted to avoid numerical instabilities. These duplicates often appeared in the empirical calculations, especially when the algorithm was close to a local or global stationary point. A variety of kernel functions $k$ were investigated, and the results weren't affected significantly, even when distorting $k(f)$ with some random noise ${X}'\sim \mathcal{U}(a,b)\forall k({{f}_{{\hat{i}}}})\in (a,b)$. 

Accordingly, for each dimension $j$ of the vector space $A$, the marginal probability density function ${P}_{{{X}_{j}}}$ is obtained numerically from the kernel function: 
\begin{equation}
    {{P}_{{{X}_{j}}}}({x}_{ij})=k(f({{x}_{ij}}))\forall j\in \{1,2,...n\}.
    \label{eq:pdf}
\end{equation}
${{P}_{{{X}_{j}}}}({{x}_{ij}})dx$ is the probability that ${\operatorname {arg\,min} }\,f$ falls within the infinitesimal interval $[{{x}_{ij}}-dx/2,{{x}_{ij}}+dx/2]$. The corresponding cumulative distribution function (CDF) ${{F}_{{{X}_{j}}}}$, can be calculated by:
\begin{equation}
    {{F}_{{{X}_{j}}}}({{x}_{ij}})=\int_{l{{b}_{j}}}^{{{x}_{ij}}}{{{P}_{{{X}_{j}}}}}({{\xi}})d\xi,
    \label{eq:cdf}
\end{equation}
where $l{{b}_{j}}$ is the lower bound of the $j^{th}$ dimension of the vector space $A$ and can be numerically evaluated by some numerical integration rule, such as the Riemann sum $S_j=\sum\limits_{i=1}^{n}{{{P}_{{{X}_{j}}}}}(x_{ij}^{*})\Delta {{x}_{ij}}$, with ${{{P}_{{{X}_{j}}}}}(x_{ij}^{*})$  computed by the application of the kernel $k$ on some ${x}_{ij}$, such that  $f(x_{ij}^{*})=(f({{x}_{ij}})+f({{x}_{i-1,j}}))/2$, or another approximation scheme. 
%The time indices $i$, indicate the sorted values of $\mathbf{x}_i$, for each dimension $j$. 
In the following pseudo-code \ref{al:ITSO}, the algorithmic implementation of the Inverse Transform Sampling method is demonstrated. The symbols are also noted in the Nomenclature section.

A graphical demonstration of the evolution of the Probability Density, as well as the corresponding Cumulative Distribution Functions, is presented in Figure
\ref{fig:cdf}, for $f(x)=(x-5)^2$, which is function with one extreme value and in Figure  \ref{fig:pdf-cdf}, for $f(x)=sin(x+0.7)+0.01*(x+0.7)^2$, which has many extrema. Interestingly, as the function evaluations increase, the CDF, is characterised by sharp slopes, which are positioned in regions where the PDF exhibits high values, and hence function $f(x)$ attends its lows. Figures \ref{fig:cdf} and \ref{fig:pdf-cdf}, offer an intuitive representation of the procedure for finding the minimum of the function, within the suggested framework.

%\vspace{5mm}
%\par

%\color{black}
%\makeatletter
%\newcommand{\removelatexerror}{\let\@latex@error\@gobble}
%\makeatother

%\newcommand{\myalgorithm}{%
%\begingroup
%\removelatexerror

\begin{algorithm}[!ht]
\label{ITSO}
\SetAlgoLined
\KwData{$A,{{f}_{e}},n$}
\KwResult{${\mathbf{x}^{*}={\operatorname {arg\,min} }\,f}$, ${{f}^{*}}=f(\mathbf{x}^*)$}
\While{$i\le {{f}_{e}}$}{
    $j\leftarrow \mathcal{U}\left\{ 1,n \right\}$\;
    $r_i \leftarrow \mathcal{U}(0,1)$\;
    % ${{x}_{ij}} \leftarrow l{{b}_{j}}+(u{{b}_{j}}-l{{b}_{j}})r$\;
    SORT $x_{ij} \hspace{2mm} \forall i$\; 
    ${{x}_{ij}}\leftarrow {F_j}^{-1}(r_i)$\; % \hspace{1cm} $\blacktriangleright$Equation \ref{eq:cdf}
    ${{f}_{i}}=f(\mathbf{x}_i)$\;
    \eIf{${{f}_{i}}\le {{f}^{*}}$}{
    ${{f}^{*}}\leftarrow {{f}_{i}} $\;
    ${{x}^{*}_j}\leftarrow {{x}_{ij}}$\;
    }{${{x}_{ij}}\leftarrow x^*_j$\;}}
\Return $\mathbf x^* = {\operatorname {arg\,min} }\,f$
\caption{ITSO-Mathematical Framework}
\label{al:ITSO}
\end{algorithm}
%\endgroup}

\par
% \vspace{\baselineskip}

\begin{figure}[!ht]   
\centering
\subfloat[PDF, $i=3$]{\includegraphics[width=0.2\textwidth, keepaspectratio]{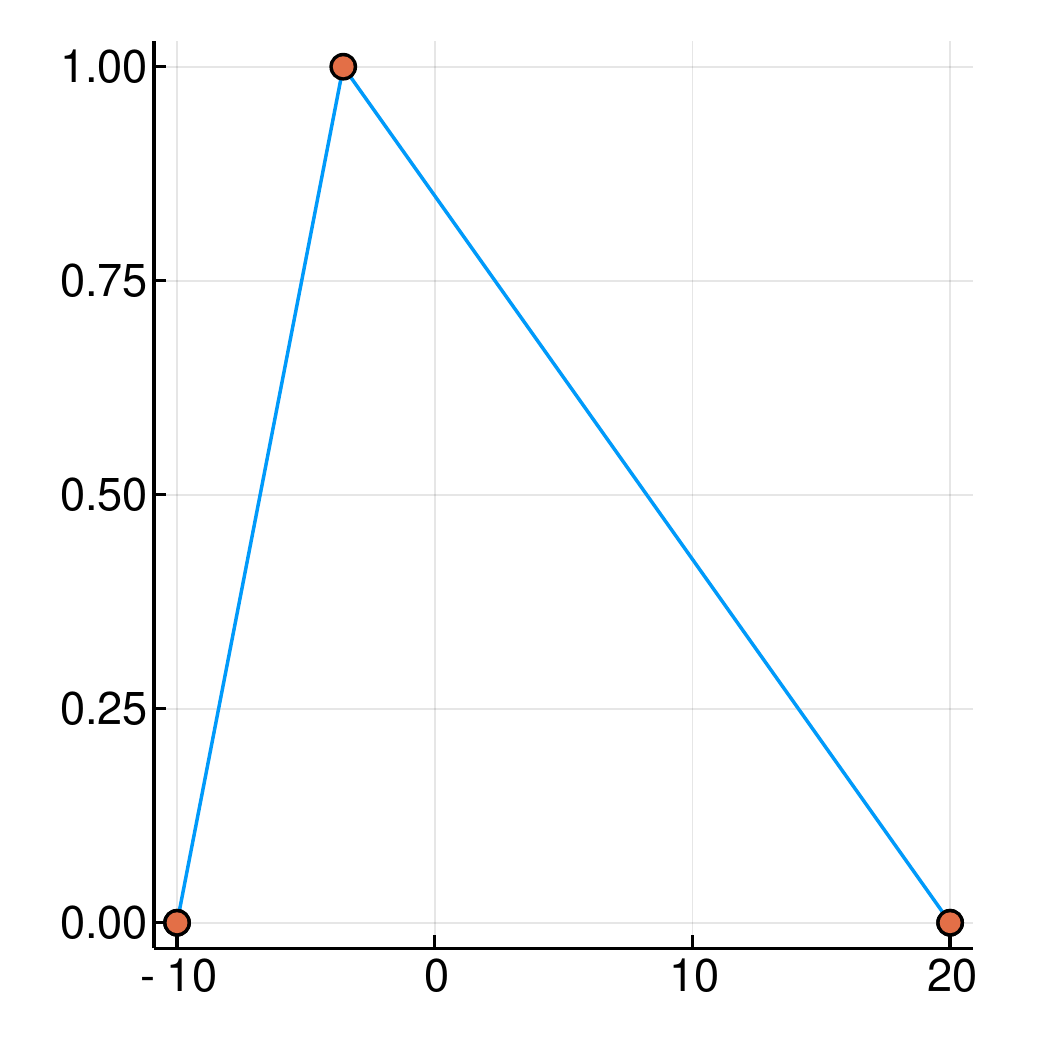}}
\subfloat[CDF, $i=3$]{\includegraphics[width=0.2\textwidth, keepaspectratio]{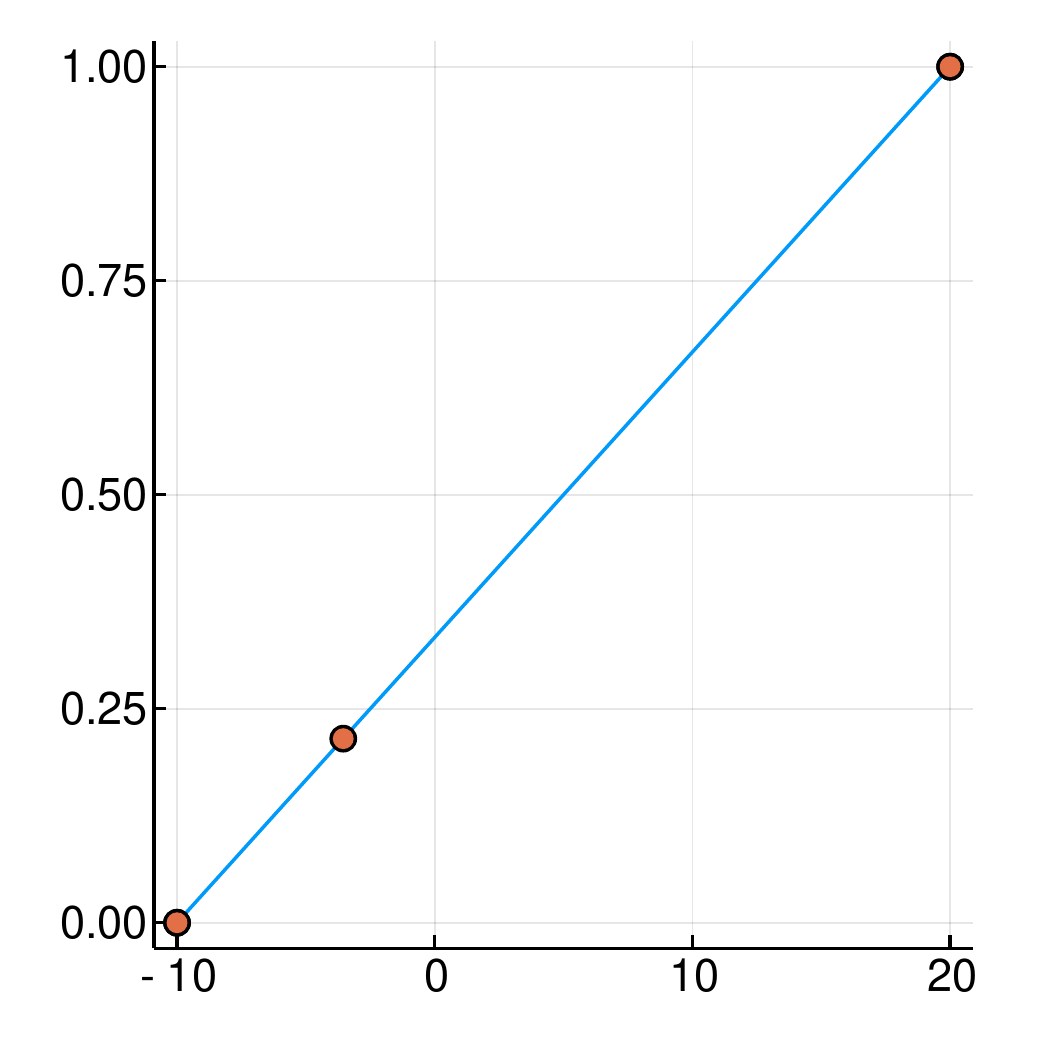}}
% \vskip\baselineskip
\subfloat[PDF,  $i=100$]{\includegraphics[width=0.2\textwidth, keepaspectratio]{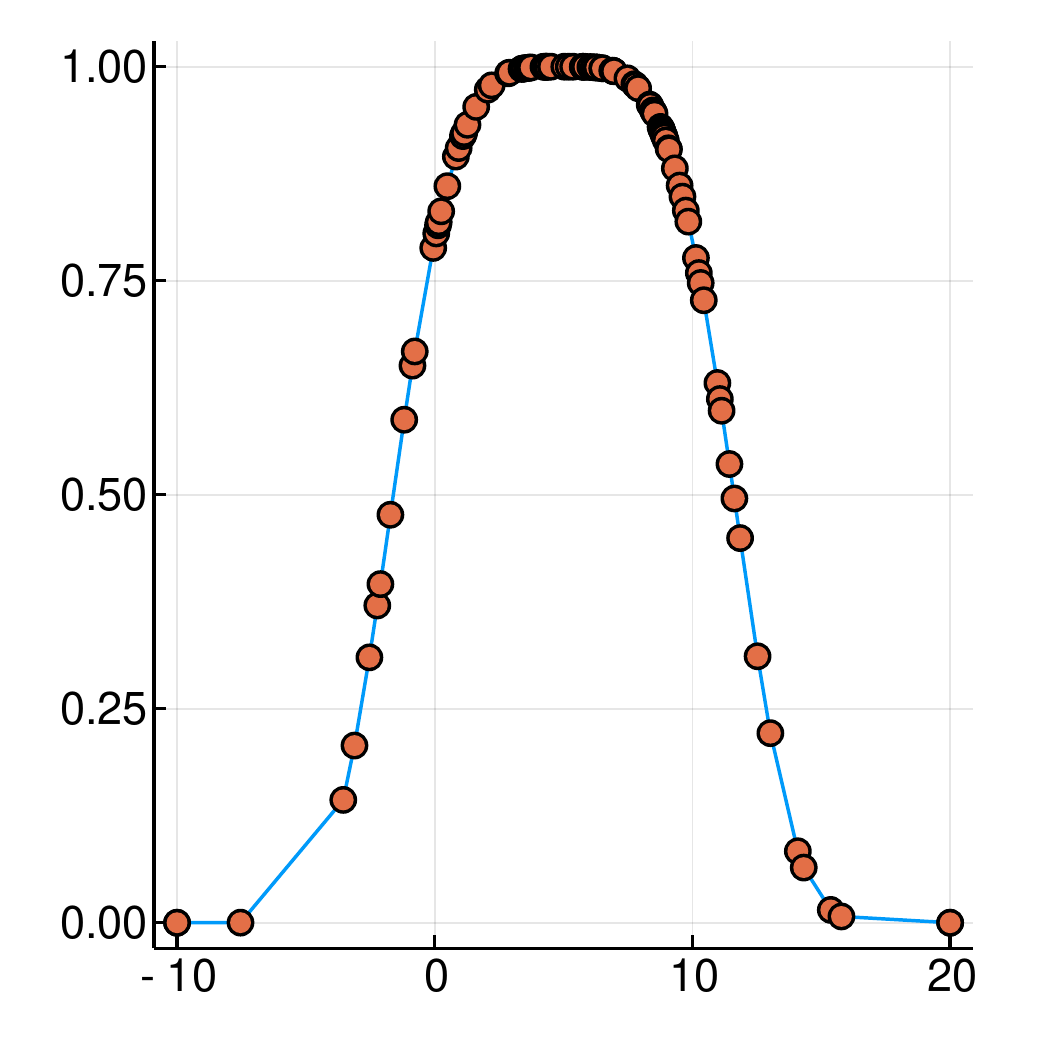}}
\subfloat[CDF,  $i=100$]{\includegraphics[width=0.2\textwidth, keepaspectratio]{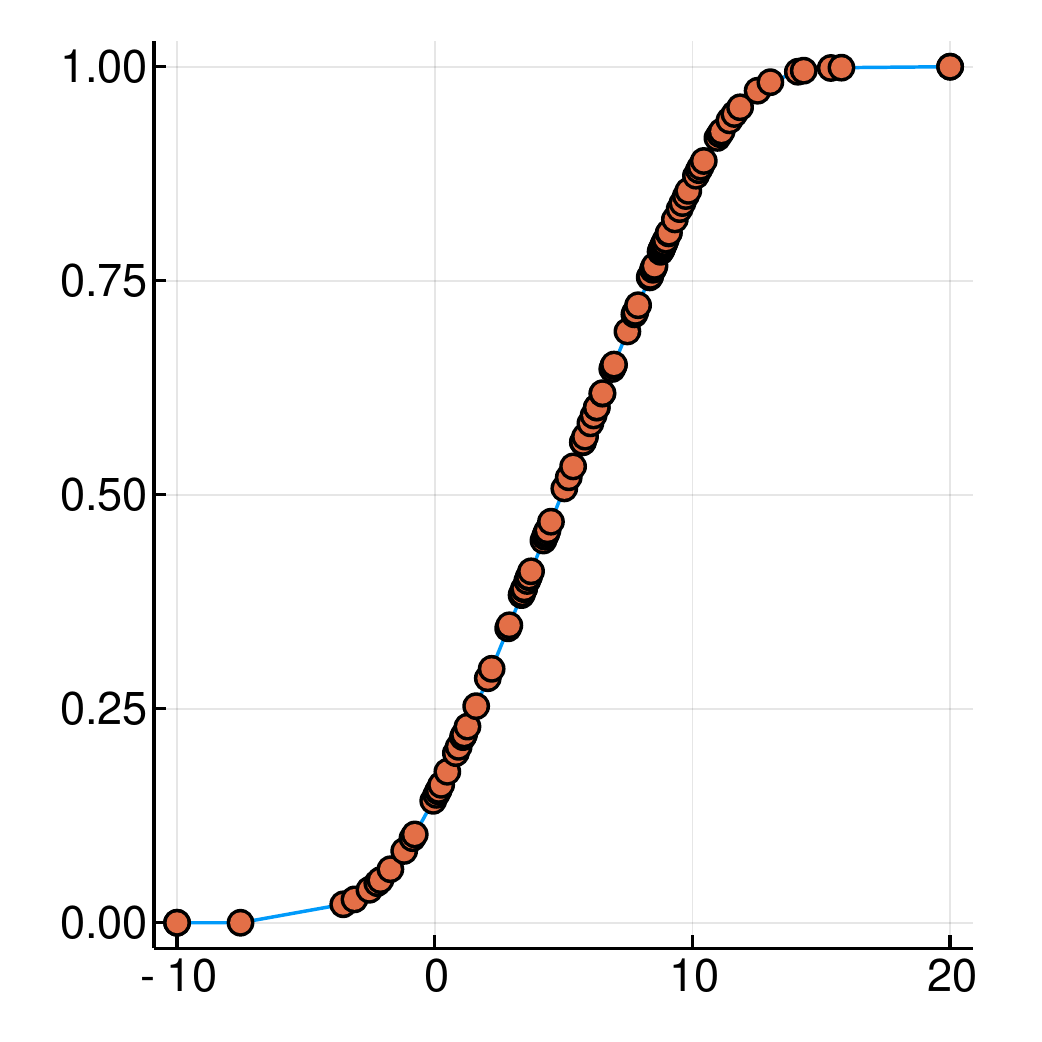}}
\vskip\baselineskip
\subfloat[PDF,  $i=350$]{\includegraphics[width=0.2\textwidth, keepaspectratio]{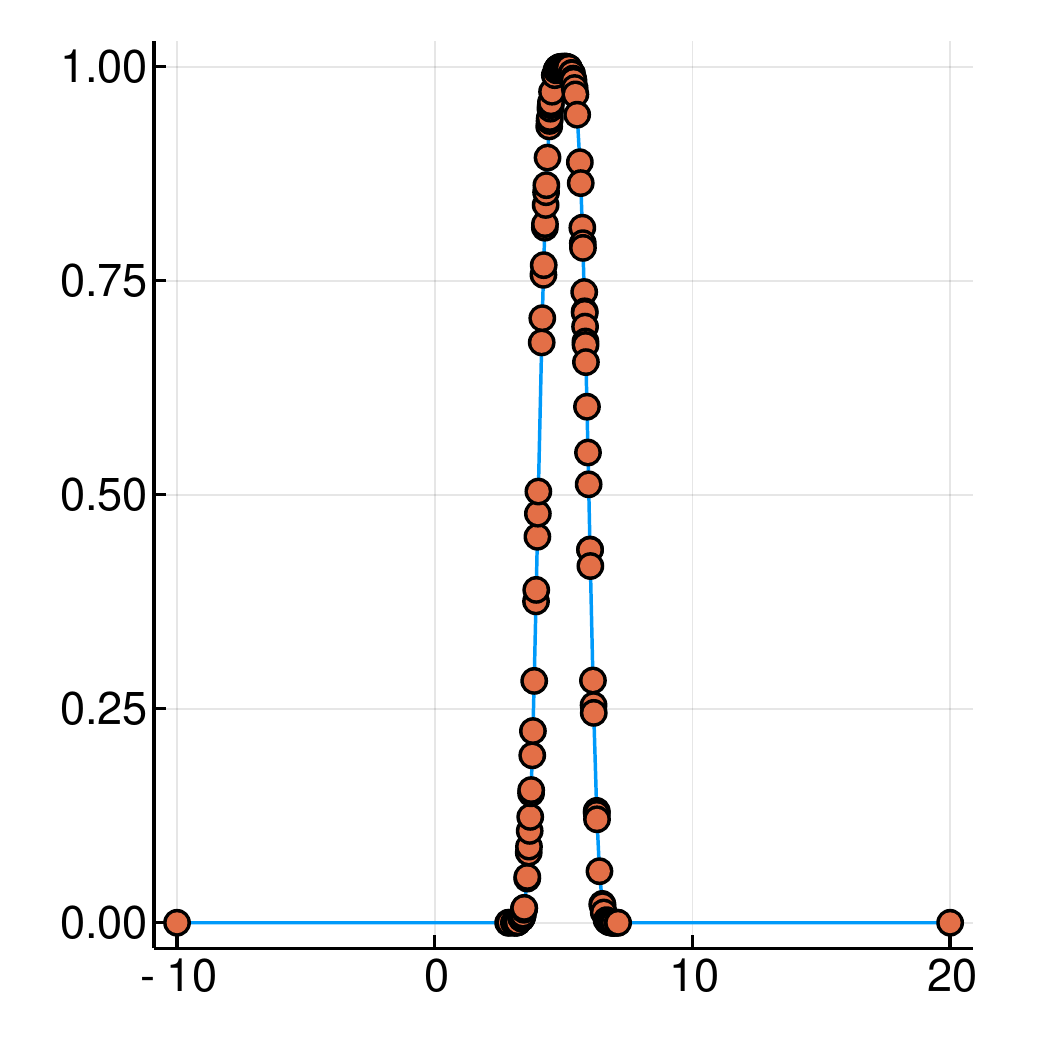}}
\subfloat[CDF,  $i=350$]{\includegraphics[width=0.2\textwidth, keepaspectratio]{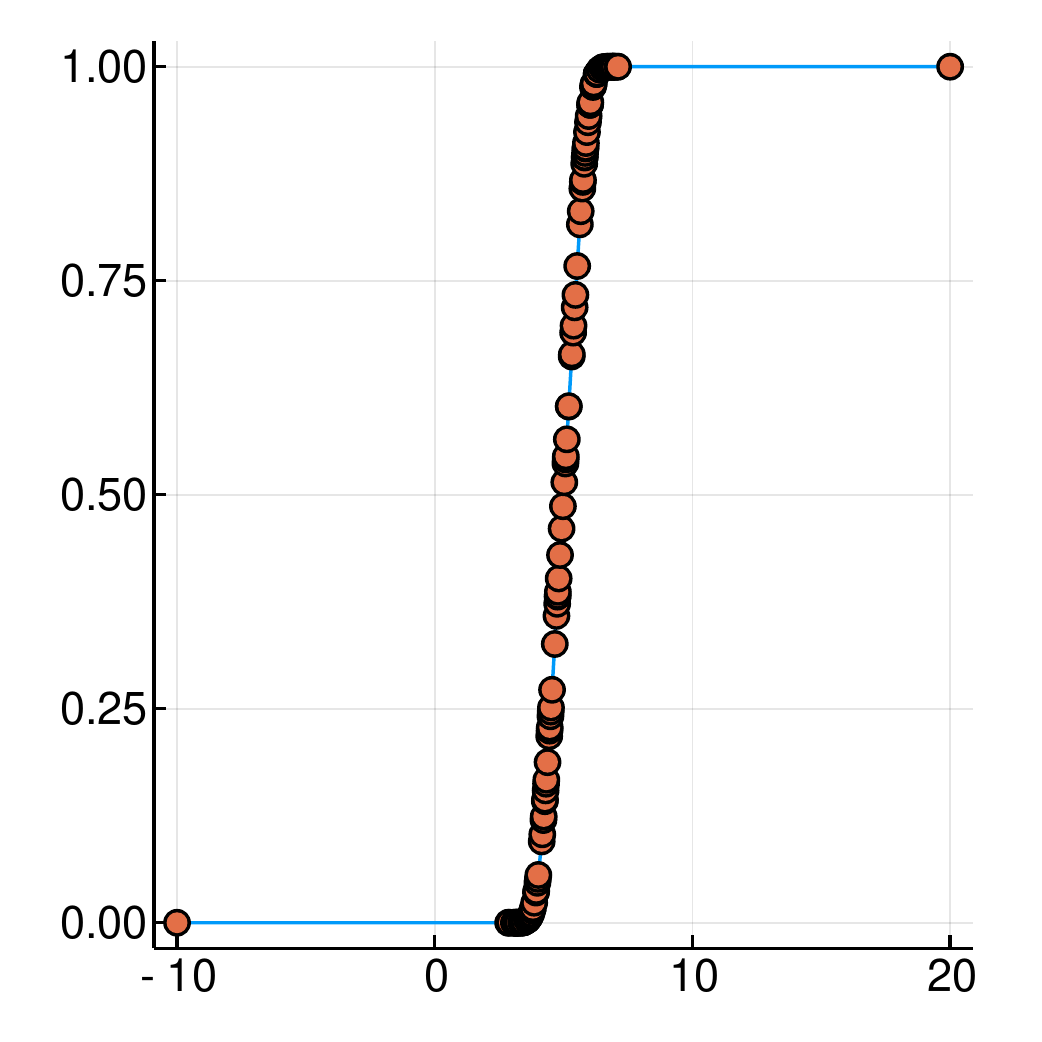}}
% \vskip\baselineskip
\subfloat[PDF,  $i=500$]{\includegraphics[width=0.2\textwidth, keepaspectratio]{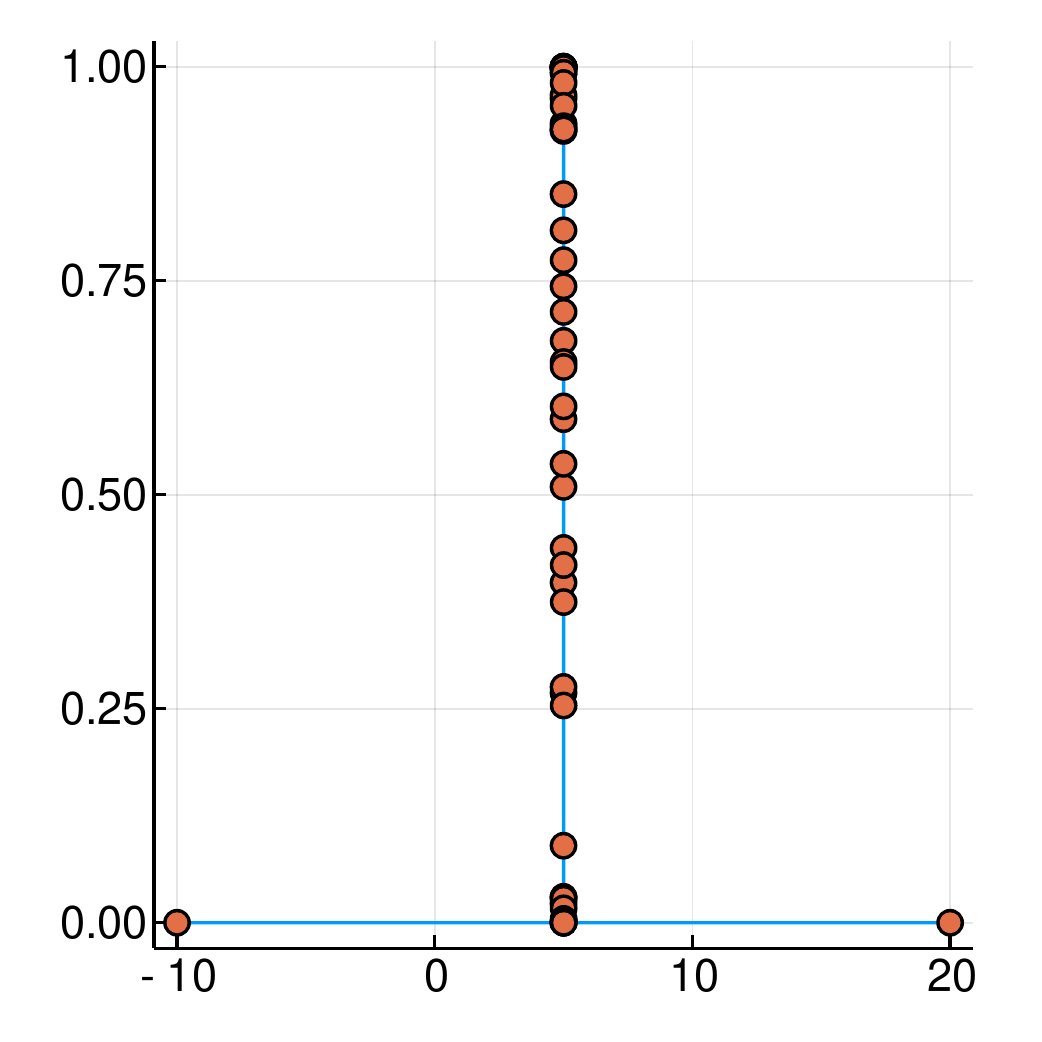}}
\subfloat[CDF,  $i=500$]{\includegraphics[width=0.2\textwidth, keepaspectratio]{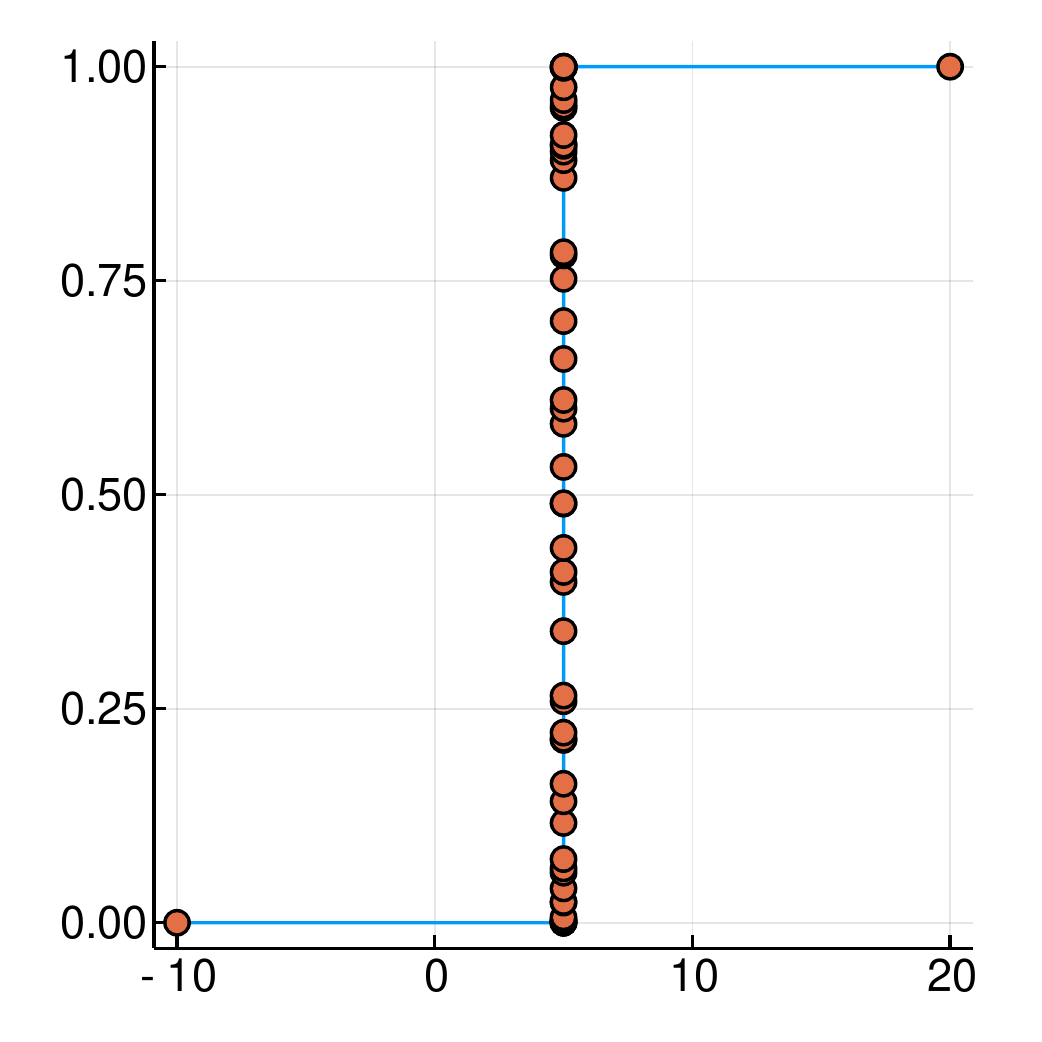}}
\caption[--]{Probability Density and Cumulative Distribution Functions, utilizing ITSO search strategy, for function $f(x)=(x-5)^2$. The shape sequentially tends to the Heaviside function, centered at $x_m=5$}
\label{fig:cdf}
\end{figure}

\begin{figure}[!ht]   
\centering
\subfloat[PDF, $i=3$]{\includegraphics[width=0.2\textwidth, keepaspectratio]{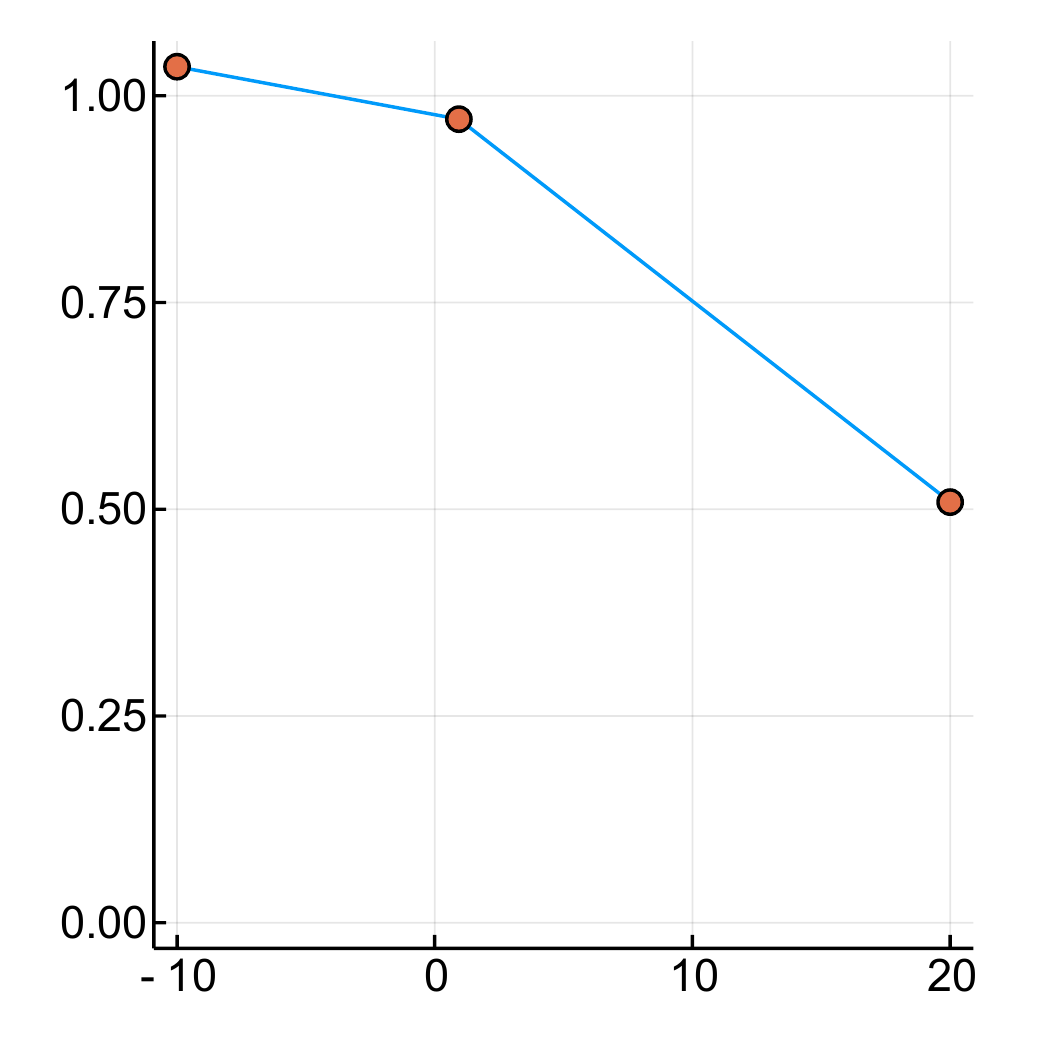}}
\subfloat[CDF, $i=3$]{\includegraphics[width=0.2\textwidth, keepaspectratio]{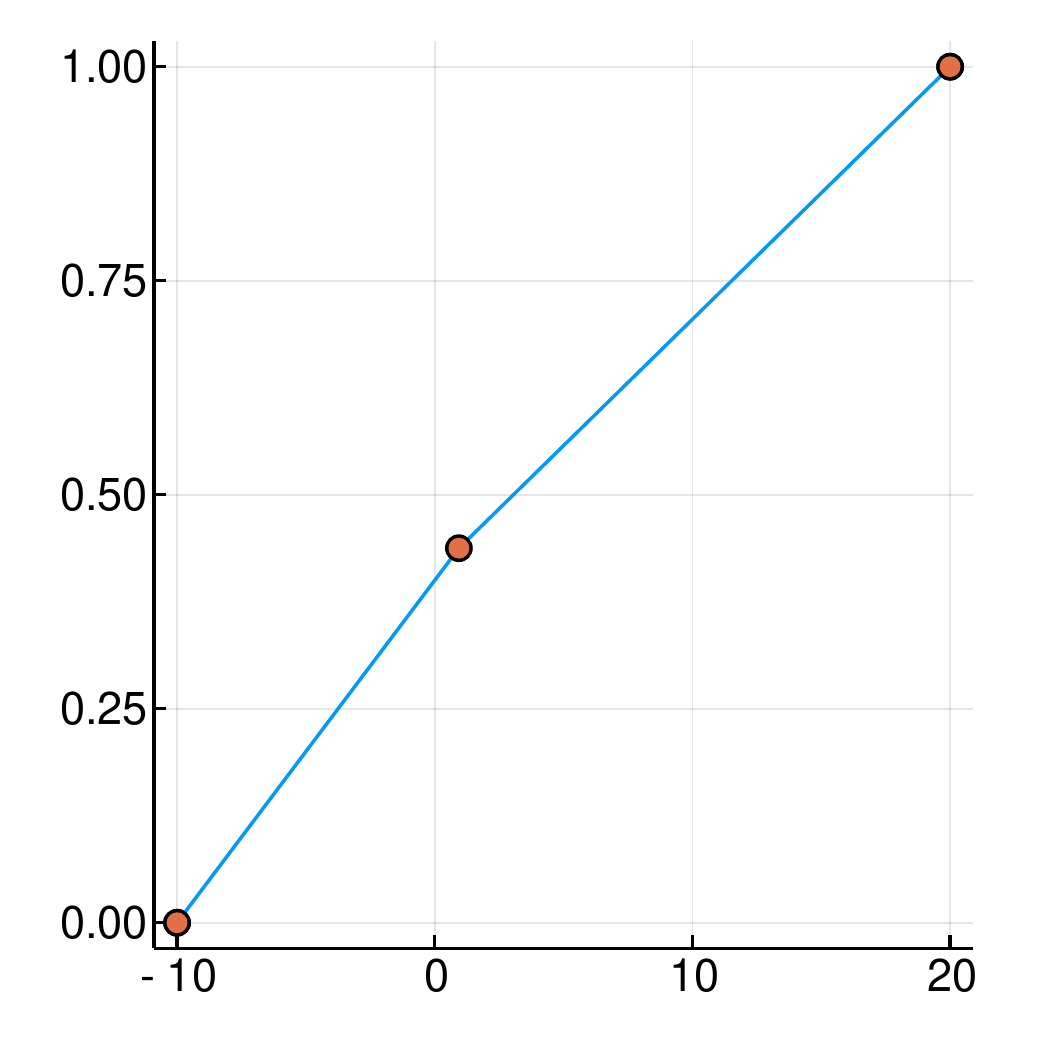}}
% \vskip\baselineskip
\subfloat[PDF, $i=150$]{\includegraphics[width=0.2\textwidth, keepaspectratio]{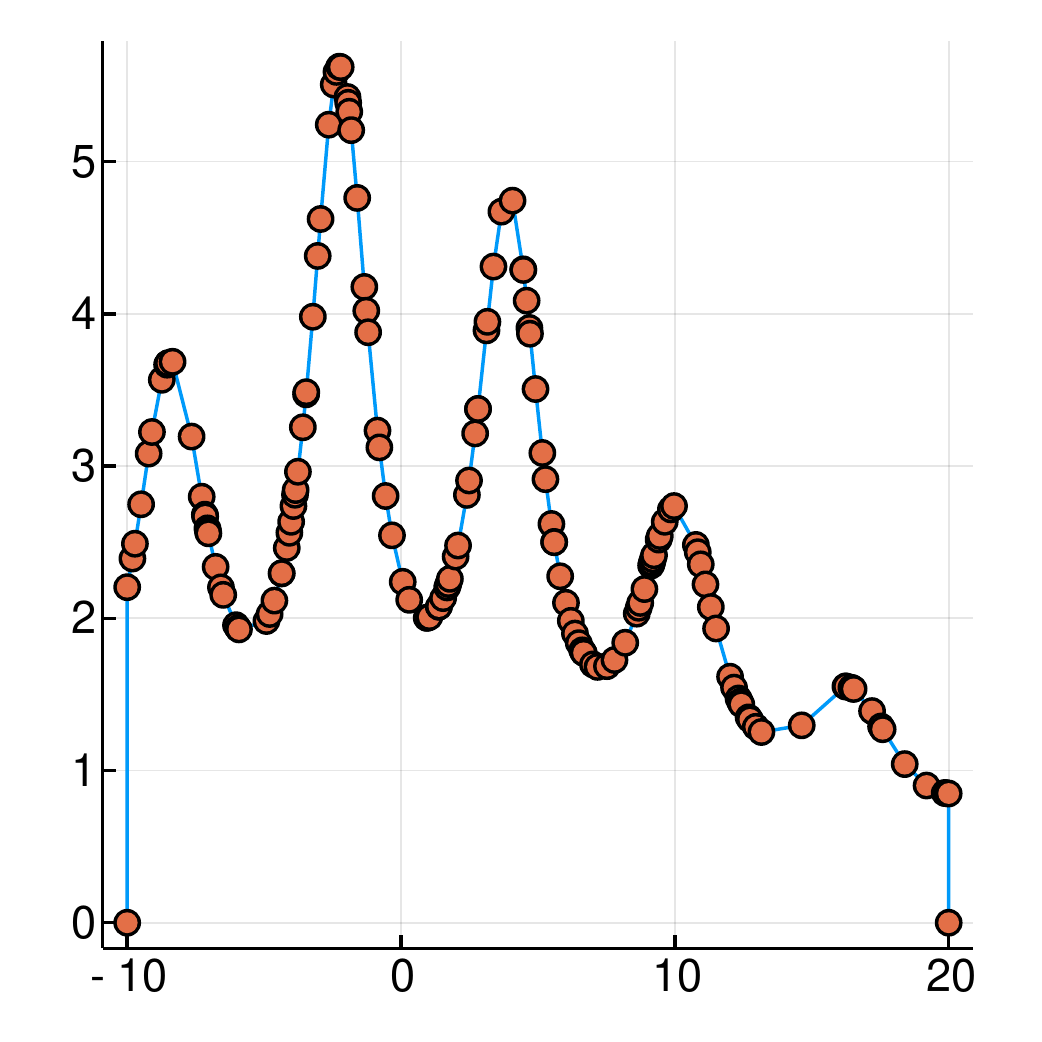}}
\subfloat[CDF, $i=150$]{\includegraphics[width=0.2\textwidth, keepaspectratio]{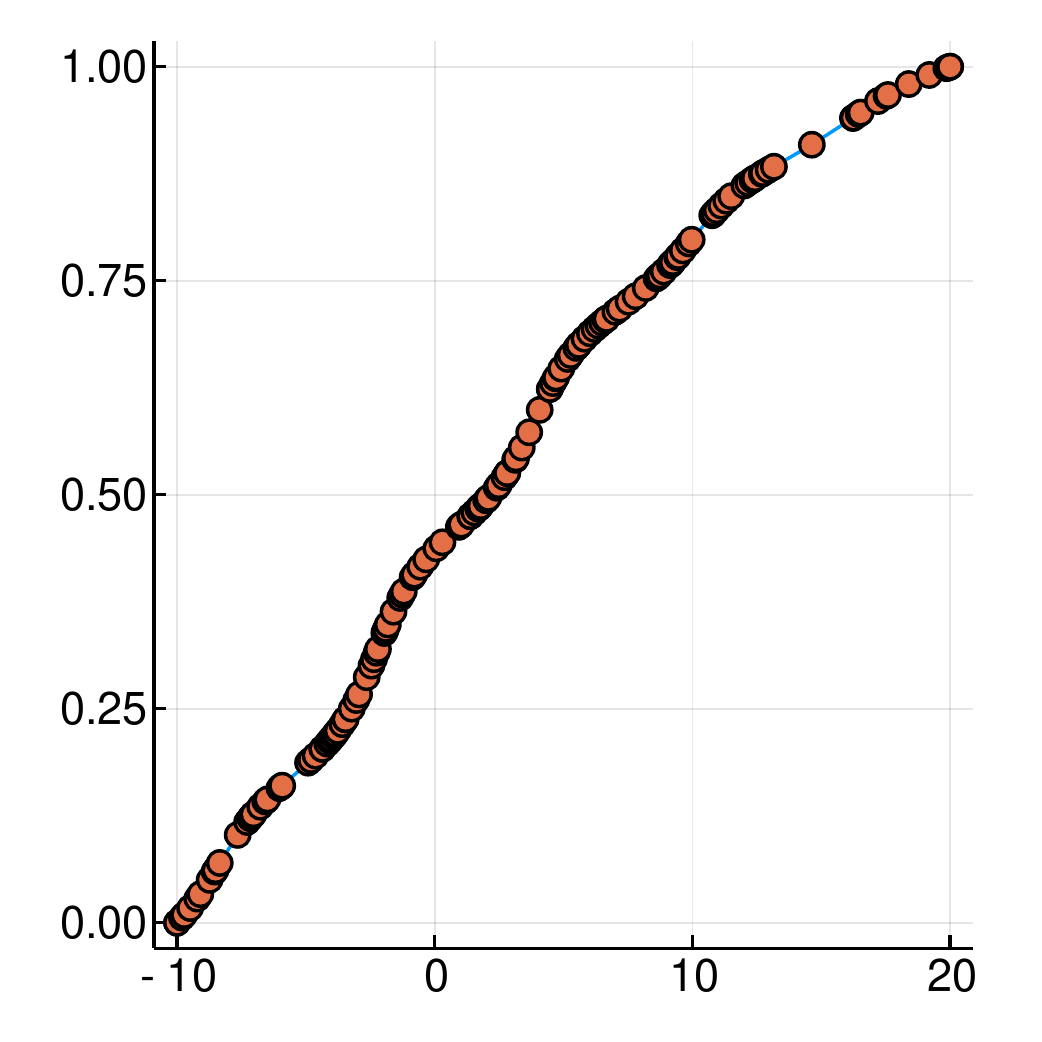}}
\vskip\baselineskip
\subfloat[PDF, $i=250$]{\includegraphics[width=0.2\textwidth, keepaspectratio]{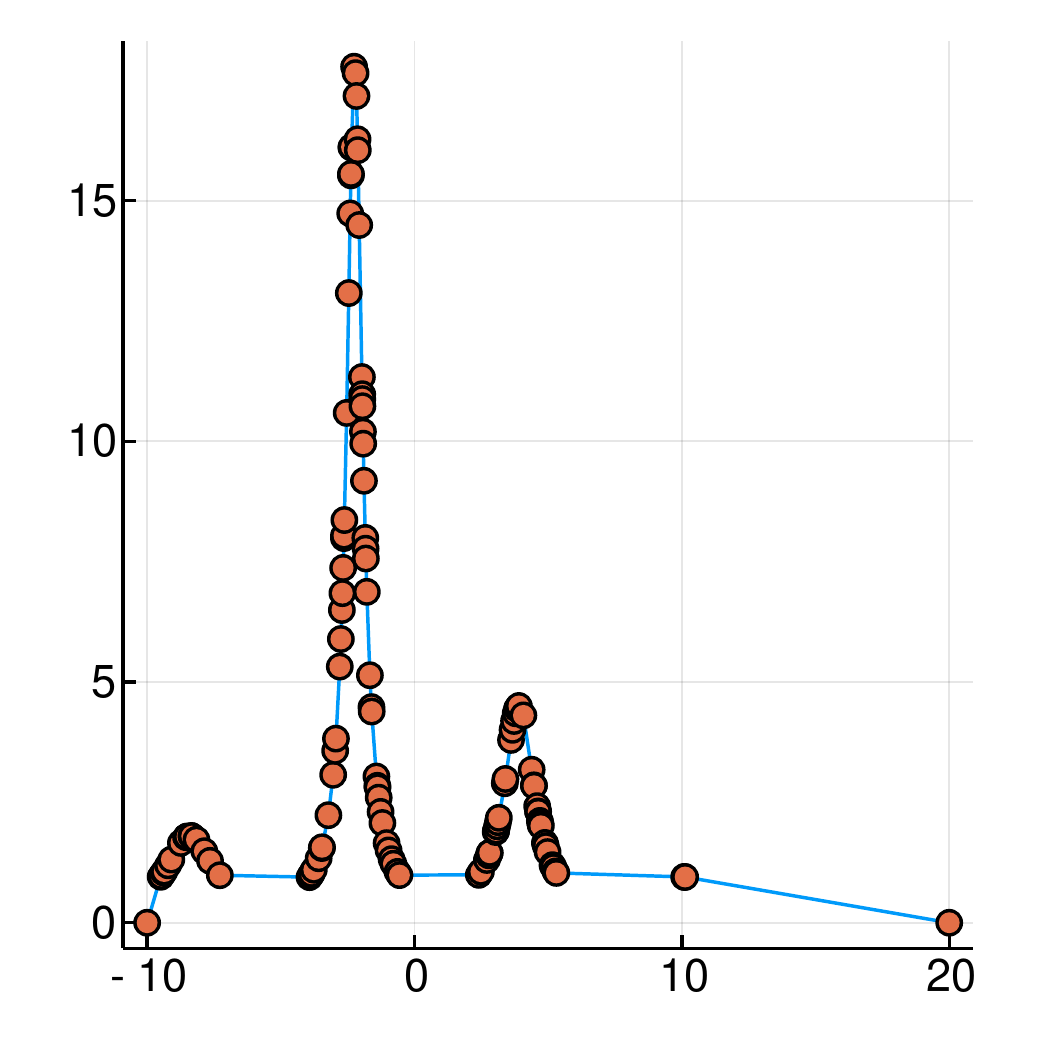}}
\subfloat[CDF, $i=250$]{\includegraphics[width=0.2\textwidth, keepaspectratio]{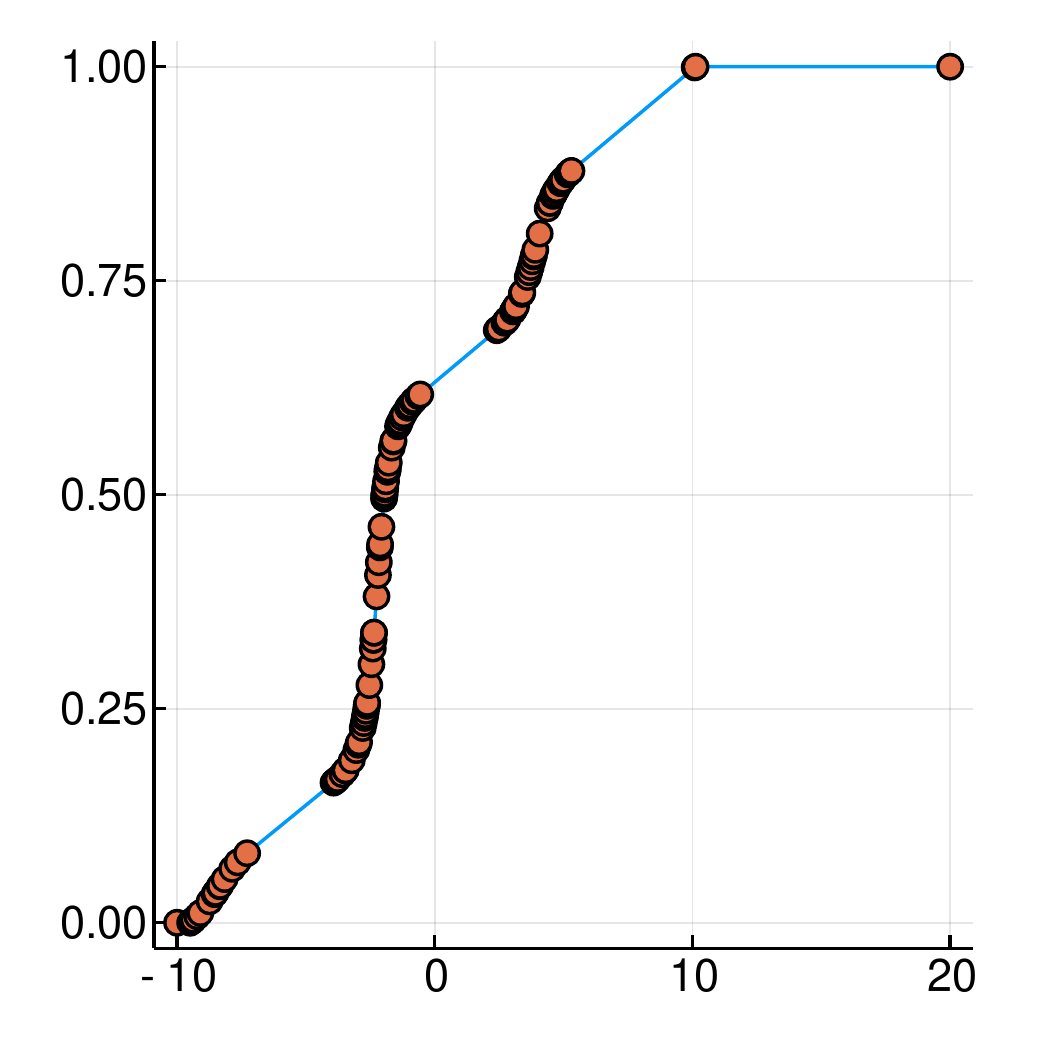}}
% \vskip\baselineskip
\subfloat[PDF, $i=500$]{\includegraphics[width=0.2\textwidth, keepaspectratio]{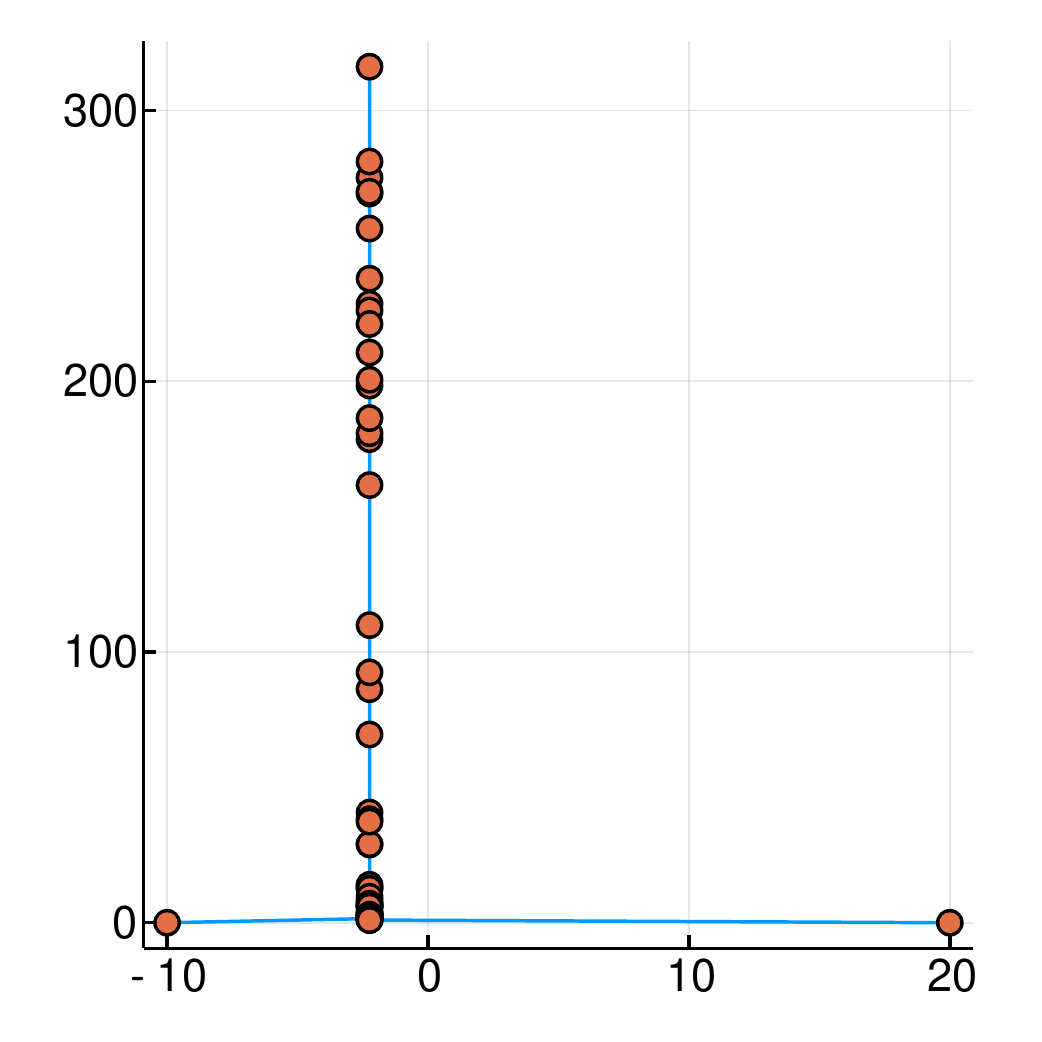}}
\subfloat[CDF, $i=500$]{\includegraphics[width=0.2\textwidth, keepaspectratio]{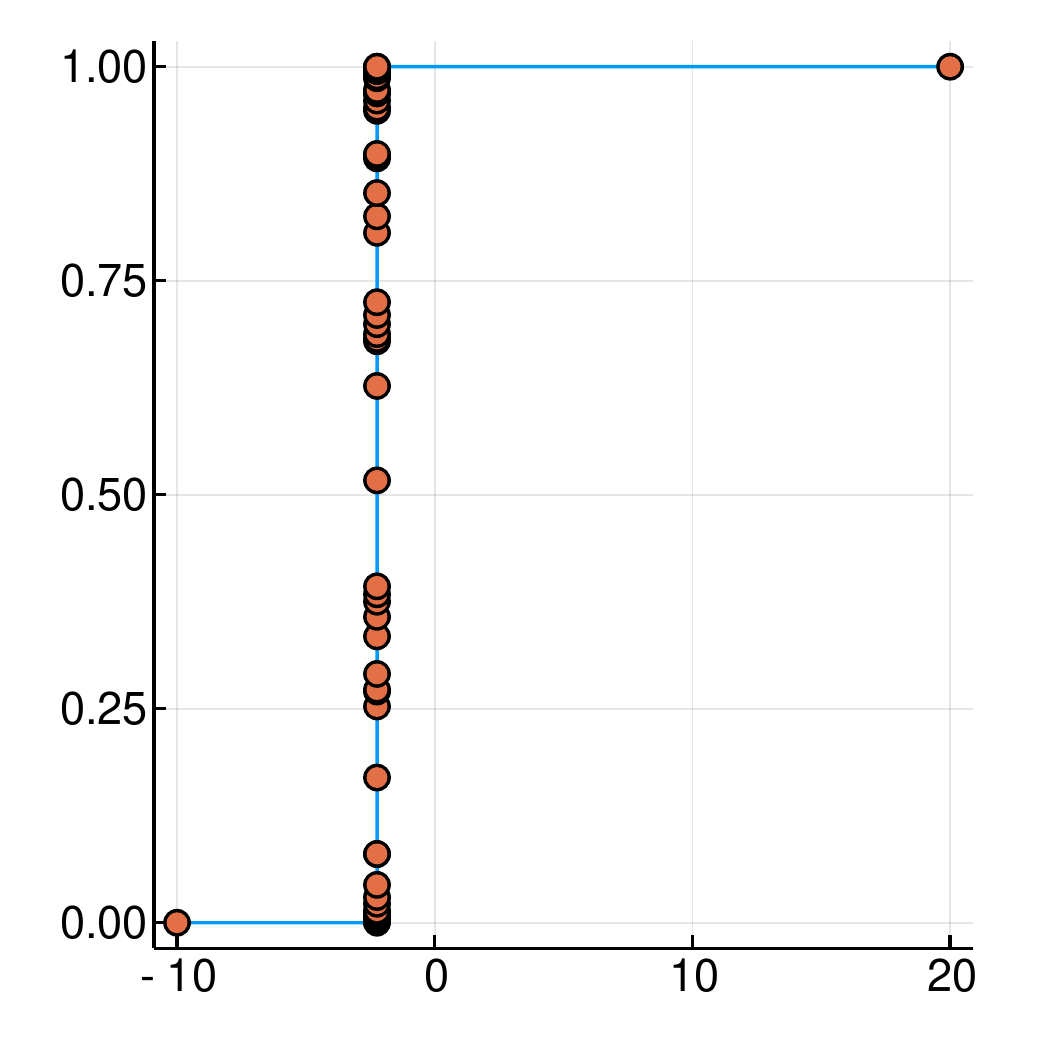}}
\caption[--]{Probability Density and Cumulative Distribution Functions, utilizing ITSO search strategy, for function $f(x)=sin(x+0.7)+0.01*(x+0.7)^2$. The shape sequentially tends to the Heaviside function, centered at $x_m=-2.24$}
\label{fig:pdf-cdf}
\end{figure}

%todo pws ginetai to ITS kai pws kaelixe o kwdikas se imi-desmevmeni pithanotita. also fi1

\section{Convergence Properties}
\label{sec:proof}

During the optimization process, the algorithm generates some input variables $x_{ij}$ as arguments for the \textit{black-box} function $f$. By utilizing the values of $f$, we may compute the values of the distribution ${{P}_{{{X}_{j}}}}(x_{ij})$, by Equation \ref{eq:pdf}, and ${{F}_{{{X}_{j}}}}({{x}_{ij}})$ by Equation \ref{eq:cdf}, for all $x_{ij}$.
% In order to perform the multivariate Inverse Transform Sampling \cite{anderson2003introduction,johnson2013multivariate} for each component $j$ separately, we should clarify if the $X_j=x_{1:i,j}$ are independent.

% \begin{theorem}
% \label{indepe}
% The components $j$ of the matrix ${{{x}}_{1:i,j}}$ are independent
% \end{theorem}
% \begin{proof}
% The unknown function $f$ returns values for each $x_j$ at time-step $i$, which are independent among all features $j$. Hence, by utilizing Equation \ref{eq:pdf}, for each iteration $i$, the probability that the optimum locates in a neighborhood of a point $\mathbf{x}^*$, is the product of the kernel $k$ over $f$. Hence, as the values of $f$ are independent for each dimension $j$, we deduce that
% \[P({{X}_{j}}={{x}_{j}}|{{X}_{k}}={{x}_{k}})=P({{X}_{j}}={{x}_{j}})\]
% \[\forall j\in \{1,2,...m\}\wedge k\in \{1,2,...m\}.\]
% \end{proof}

\begin{definition}
We define the argument of $f$ within $f_e$ iterations, corresponding to the minimum of $f$ as $x_m={\operatorname {arg\,min} }\, f$.
%within the obtained values during iterations?
\end{definition}

In iteration $i$, the algorithm will have searched the space $\hat{A}\subset A$, within the limits $lb_j,ub_j$.
% and the expected value of the ${\operatorname {arg\,min} }\,f$ will be equal to $E[X]=\int_{{\hat{\mathbb{A}}}}{x}P(x)dx$. This can be considered as correct enough for the final limit state of search, where small changes occur in $P$ and can be considered as constant for a few iterations. 
By definition, $P$ is the probability density (likelihood) that the optimum occurs in a region $\mathbf{x}\pm \mathbf{dx}$. Hence,
% from Theorem \ref{indepe}, 
%the $E[X]$ applies for each dimension of the vector space $A$. For a multivariate function, the  expected value of $X$ over execution time will be 
\begin{equation}
E[\mathbf X]=
\begin{Bmatrix}
E[X_1]\\
E[X_2]\\
\hdots \\
E[X_n]
\end{Bmatrix}=
\begin{Bmatrix}
\int_{lb_1}^{ub_1}{\xi}P_{X_1}(\xi)d\xi\\
\int_{lb_2}^{ub_2}{\xi}P_{X_2}(\xi)d\xi\\
\hdots \\
\int_{lb_n}^{ub_n}{\xi}P_{X_n}(\xi)d\xi
\end{Bmatrix},
\label{eq:pdf-all}
\end{equation}

where $E[\cdot]$, denotes the expectation of a random variable or vector.
%\begin{definition}
%From Equation \ref{eq:pdf-all}, for each current dimension $j$ of Algorithm \ref{al:IIS}, we may write 

%\begin{equation}
 %   E[X_j]=\int_{lb_j}^{ub_j}{x}P(x)dx,
  %  \label{eq:expect}
%\end{equation}

%where $P$ is the marginal PDF. 
%\end{definition}

\begin{theorem}
\label{ITSOth}
If $P_{X_j}$ tends to Dirac $\delta_m$, then ITSO will converge to $x_m$
\end{theorem}
\begin{proof}
For each dimension $j$, we may write
\begin{equation}
E[X_j]=\int_{lb_j}^{ub_j}{\xi}P_{X_j}(\xi)d\xi=\int_{lb_j}^{ub_j}{\xi}{F'_{X_j}}(\xi)d\xi,
\end{equation}
and integrating by parts, we obtain
\begin{equation}
\begin{split}
    E[X_j]=[\xi F_{X_j}(\xi)]_{lb_j}^{ub_j}-\int_{lb_j}^{ub_j}{1}F_{X_j}(\xi)d\xi=ub_j*1-lb_j*0
    \\
    -\int_{lb_j}^{ub_j}{1}F_{X_j}(\xi)d\xi.
    \label{eq:expect}
\end{split}
\end{equation}
If we apply the theorem of the antiderivative of inverse functions \cite{Parker1955}, we obtain
\begin{equation}
    \int_{0}^{1}{{{F_{X_j}}^{-1}}}(y)dy+\int_{lb_j}^{ub_j}{F_{X_j}}(x)dx=ub_j*1-lb_j*0.
    \label{eq:antider}
\end{equation}

Hence, for Equations \ref{eq:expect} and \ref{eq:antider} we deduce that 
\begin{equation}
    E[X_j]=\int_{0}^{1}{{{F_{X_j}}^{-1}}}(y)dy.
\label{eq:f-1}
\end{equation}

With subscript $m$ denoting that the Dirac function is centered at the argument that minimizes $f$, i.e.
\begin{equation}
{{\delta }_{m}}(x)=\left\{ \begin{array}{*{35}{l}}
   +\infty , & x={\operatorname {arg\,min} }\,f  \\
   0, & x\ne {\operatorname {arg\,min} }\,f  \\
\end{array} \right. ,
\end{equation}
and
\begin{equation}
H_m(x):=\int_{-\infty }^{x}{\delta_m (s)}ds.
\end{equation}

As $P_{X_j}$ tends to Dirac function, $F_{X_j}$ tends to the Heaviside step function centered at $x_m$, hence by Equation \ref{eq:f-1} we deduce that
\begin{equation}
\label{eq:ex}
    E[X_j] \xrightarrow[]{i\to f_e} \int_{0}^{1}{{{H_m}^{-1}}}(y)dy,
\end{equation}

and by Equation \ref{eq:antider} 
\begin{equation}
    \label{eq:xm}
    E[X_j] \xrightarrow{} ub_j - \int_{0}^{1}{{H_m}}(y)dy = ub_j - (ub_j-x_m)*1 = x_m
\end{equation}
\end{proof}

\begin{lemma}
\label{lemma-conv}
\emph{(ITSO Convergence Speed)}
ITSO is the \underline{fastest possible} optimization framework
\end{lemma}
\begin{proof}
Any distribution that doesn't tend to Dirac, could be considered as a Dirac plus a positive function of $x$. In this case, the algorithmic framework would search through a strategy that produces a ${P}'$ over $\mathbf{x}$, as 
\begin{equation}
    \label{eq:p*}
    {P}'={{P}^{*}} \pm {{\delta }_{m}}.
\end{equation}

Hence, with $F^*$ indicating the CDF corresponding to $P^*$, Equations \ref{eq:ex}, and \ref{eq:f-1}, result in

\begin{equation}
E[X] \xrightarrow[]{i\to f_e} \int_{0}^{1}{{{H_m}^{-1}}}(y)dy + \int_{lb}^{ub}{F^*(x)dx},
\end{equation}. 

and thus, by Equations \ref{eq:p*}, and \ref{eq:xm}, we obtain
\begin{equation}
E[X] \xrightarrow{} x_m \pm\epsilon.
\end{equation}

Hence the algorithm would converge to a point different than $x_m$

% integral in Equation \ref{eq:f-1} will always have a remainder $\epsilon >0$ hence the minimum will be ${\operatorname {arg\,min} }\,f +e>0$. 

% Assume that the best possible $P$ to assign is ${{P}_{b}}$, then $E[X]=\int_{{\hat{\mathbb{A}}}}{x}(P(x)-{{P}_{b}}(x))dx$.

\end{proof}
% --for the same number of iterations
% --for example random the expectd = mean of bounds
% --this stands for infinite iterations, without selection teleio edw stekei, me selection pening
% --se ena finite number tha einai ali to kalyero ews tote hwris selection

\section{Programming techniques}
\label{sec:prog}

A variety of programming techniques were investigated, in order to implement the proposed method into a computer code. %Initially, a linearization of the PDF within the computed $\mathbf x_i$ was implemented, however in order to simplify the algorithm, and hence reduce the computational time, 
To keep the algorithm simple and reduce the computational time, we applied inverse transform sampling by keeping in each iteration $i$ the best function evaluations, and randomly sampling among them. This is equivalent to a kernel function that vanishes over the worst function evaluations, and distributing the probability mass to the best performing ones. Accordingly, similar programming techniques may be investigated in future works, within the suggested framework.

\begin{algorithm}[!ht]
\SetAlgoLined
\textbf{Initialize}: $\mathbf{x}=\text{rand}(n), \quad \mathbf{opti\_x}=\mathbf{x}, \quad {opti\_f}=f(\mathbf{\mathbf{opti\_x}})$\;
\For{$i=1:f_e$}{
    \For{$j=1:n$}{
        $\mathbf{indsBEST}=\text{sortperm}({opti\_evals})[1:\alpha]$\;
        $\mathbf{inp\_x=x\_evals}[\mathbf{indsBEST},j]$\; $\mathbf{inp\_y=opti\_evals}[\mathbf{indsBEST}]$\;
        ${rr}= \text{min}\{\mathbf{inp\_x\}}$ \\
        $\quad + \text{rand}(0,1)(\text{max}\{\mathbf{inp\_x\}} - \text{min}\{\mathbf{inp\_x\}})$\;
        $\mathbf{x}[j]={rr}$\;
        $f_i=f(\mathbf{x})$\;
        \eIf{${{f}_{i}}\le$ \text{opti\_f}}{
        ${opti\_f} = {{f}_{i}} $\;
        $\mathbf{opti\_x} = \mathbf{x}$\;
        }{$\mathbf{x}=\mathbf{opti\_x}$\;}
    }
}
\Return $\quad {opti\_f}, \quad \mathbf{opti\_x}$
\caption{ITSO-Short Pseudocode}
\label{ITSO-short}
\end{algorithm}

The Algorithm \ref{ITSO-short} represents a simple version of the supplementary code in the appendix, which may easily be programmed. The variable $\mathbf{opti\_evals}$ is a dynamic vector, containing all values of the objective function returned during the optimization history, until step $i$, and $\mathbf{x\_evals}$ is an $i \times n$ matrix, containing all the design vectors $\mathbf{x}_{1:i}$, corresponding to $\mathbf{opti\_evals}$. The integer $\alpha$ is a parameter regarding how many instances of the optimization history are kept in order to randomly sample among them; for example if $f_e=10^4$, we may select $\alpha=10^2$. In this variation of the code, the Inverse Transform Sampling is approximately implemented in line 7, by calculating the random variable $rr$, among the extrema of the vector $\mathbf{inp\_x}$, corresponding to the range where: for the $j^{th}$ dimension of all $\mathbf{x}_{1:i}$, the $\alpha$ best function values were returned by the \textit{black-box} function $f$. The sought solution is the vector $\mathbf{opti\_x}$, and the mimimum attained value of the objective function $opti\_f$. 

\section{Numerical Experiments}
%\subsection{Black-Box Functions}
\label{sec:bbf}
In this section, we present the results obtained by running the ITSO algorithm, as well as Adaptive Differential Evolution (rand 1 bin), Differential Evolution (rand 1 bin), and   Differential Evolution (rand 2 bin) with and without radius limited, Compass Coordinate Direct Search, Probabilistic Descent Direct search, Random Search, Resampling Inheritance Memetic Search, Resampling Memetic Search, Separable Natural Evolution Strategies, Simultaneous Perturbation Stochastic Approximation, and Exponential Natural Evolution Strategies from the Julia Package BlackBoxOptim.jl \cite{BlackBoxOptim}, and Nelder-Mead, Particle Swarm, and Simulated Annealing from Optim.jl \cite{Optim}. To demonstrate the performance of each optimizer in attaining the minimum, we firstly run the algorithm $r=10$ times, obtain $f_k(\mathbf{x}_i)$ for $k=\{1,2,\dots,r\}$ and all iterations $i=\left\{ 1,2,\ldots ,{{f}_{e}} \right\}$, and average the results 
\begin{equation}
    \label{eq:aner}
    \hat{f}(\mathbf{x}_i)=\frac{\sum_{k=1}^{r} f_k(\mathbf{x}_i)}{r}.
\end{equation}
Then, we normalize the vector of obtained function evaluations $\mathbf{v}=\{\hat{f}(\mathbf{x}_1),\hat{f}(\mathbf{x}_1),\dots,\hat{f}(\mathbf{x}_{f_e})\}$ in the domain $\left[0,1\right]$ through
\begin{equation}
    \label{eq:norm}
    \hat{f}_n(\mathbf{x}_i)=\frac{\hat{f}(\mathbf{x}_i)-\operatorname {min} \mathbf{v}}{\operatorname {max} \mathbf{v} -\operatorname {min} \mathbf{v}},
\end{equation}
and finally use the optimization history $h$ 
\begin{equation}
\label{eq:h}
    h(\mathbf{x}_i)=\left(\frac{\sum\limits_{l=1}^{m}{\hat{f}_n^l(\mathbf{x}_i)}}{m}\right)^{\frac{1}{10}},
\end{equation}
where $m$ indicates the number of \textit{black-box} functions used for the evaluation. Equation \ref{eq:h} was selected as a performance metric, in order to obtain a clear representation of the various optimizers utilized, as powers smaller than one (in our case $\frac{1}{10}$), have the property to magnify the attained values at the final steps of the optimization history. In Figure \ref{fig:run_tests} the numerical experiments for $m=13$ functions are presented. Each line corresponds to the normalized average optimization history $h$ (Equation \ref{eq:h}, for all functions, which were repeated 10 times). We may see a clear prevalence of the proposed framework, in terms of convergence performance, as expected by the theoretical investigation. The numerical experiments for the comparison with other optimizers, may be reproduced by running the file \textit{\_\_run.jl}.

\begin{figure*}[!ht]   
\centering
\subfloat[10 variables and 5000 evaluations]{\includegraphics[width=0.5\textwidth, keepaspectratio]{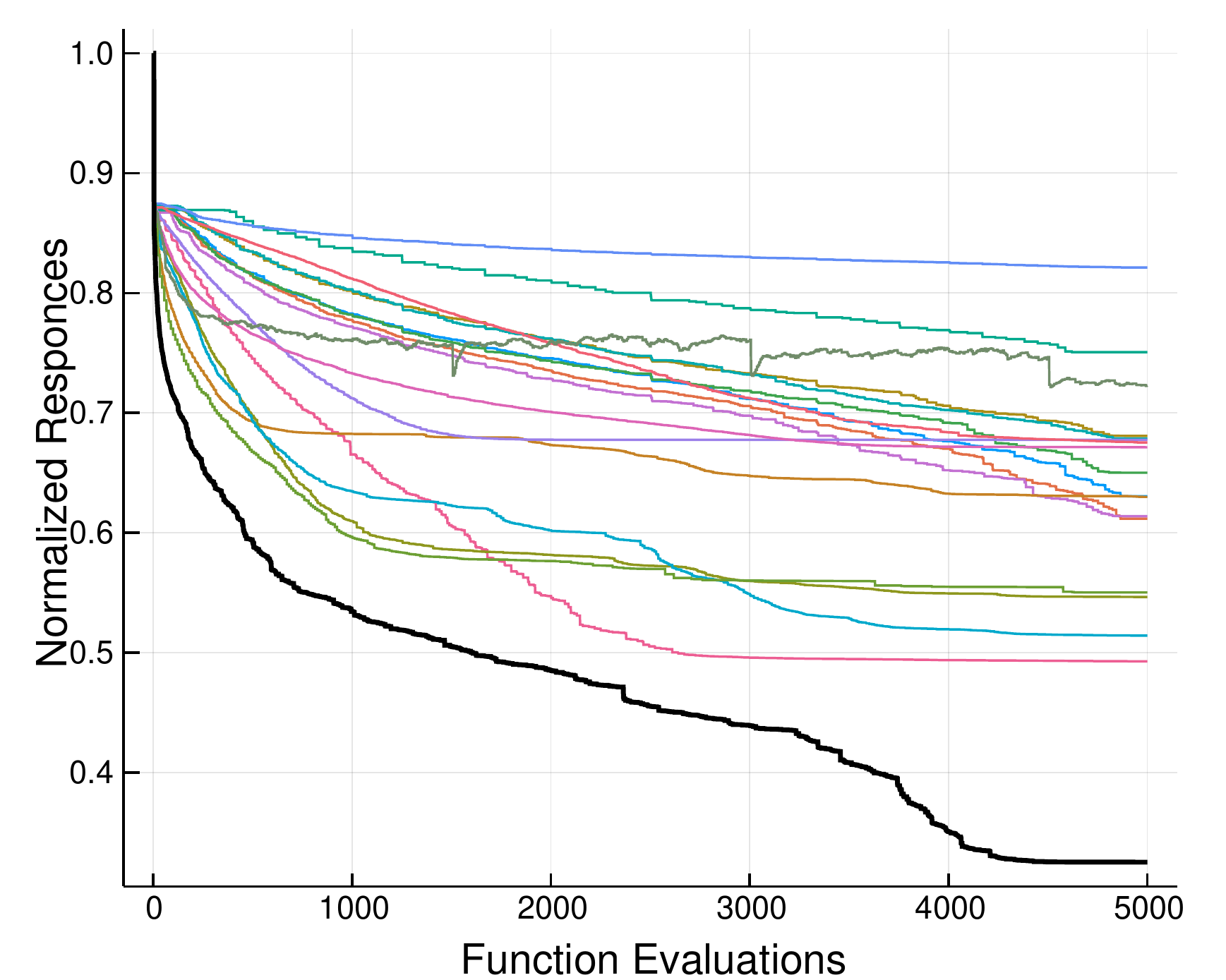}\label{fig:all_10times_10vars_5000evals}}
\subfloat[20 variables and 5000 evaluations]{\includegraphics[width=0.5\textwidth, keepaspectratio]{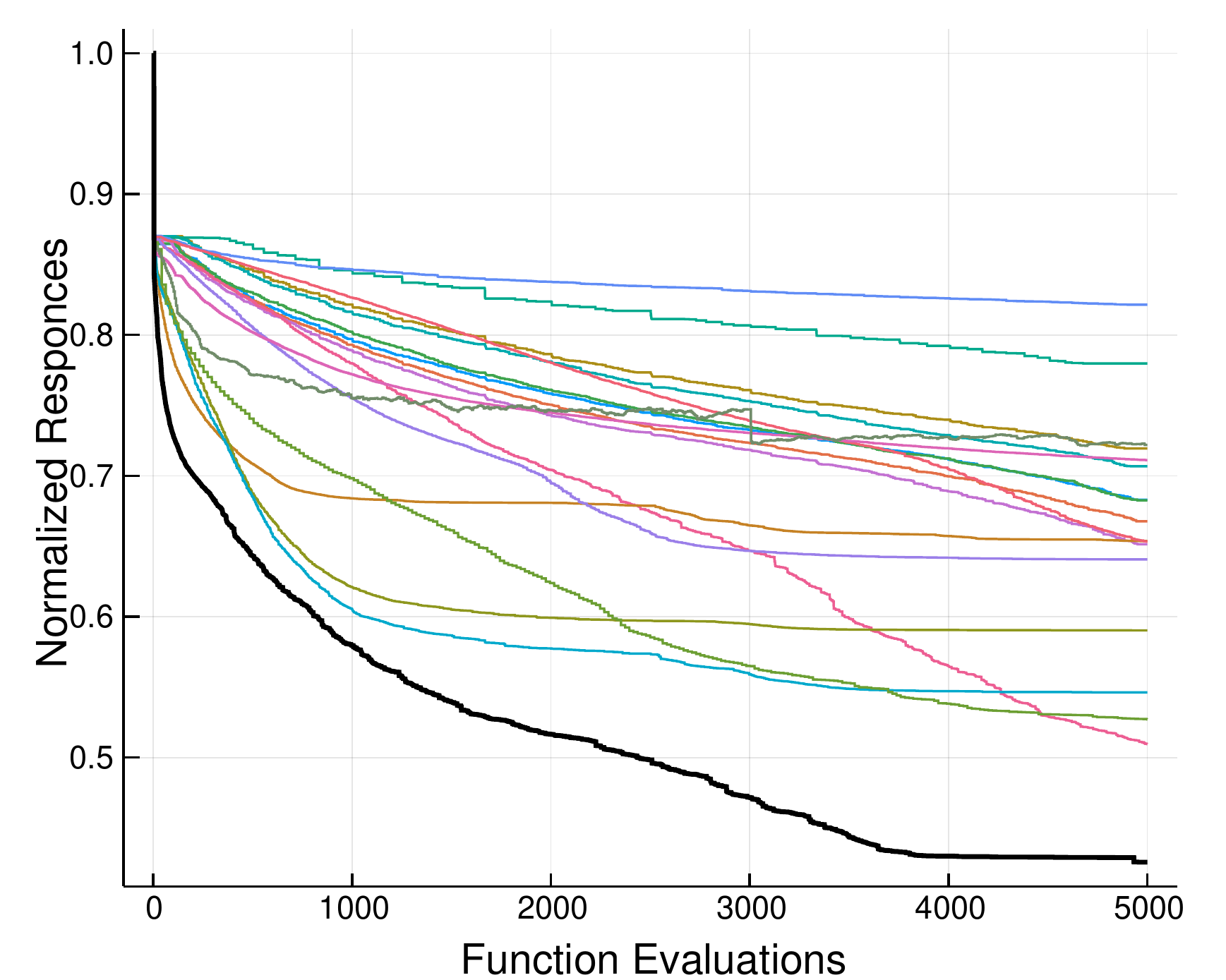}\label{fig:all_10times_20vars_5000evals}}
\vskip\baselineskip
\subfloat[20 variables and 10000 evaluations]{\includegraphics[width=0.5\textwidth, keepaspectratio]{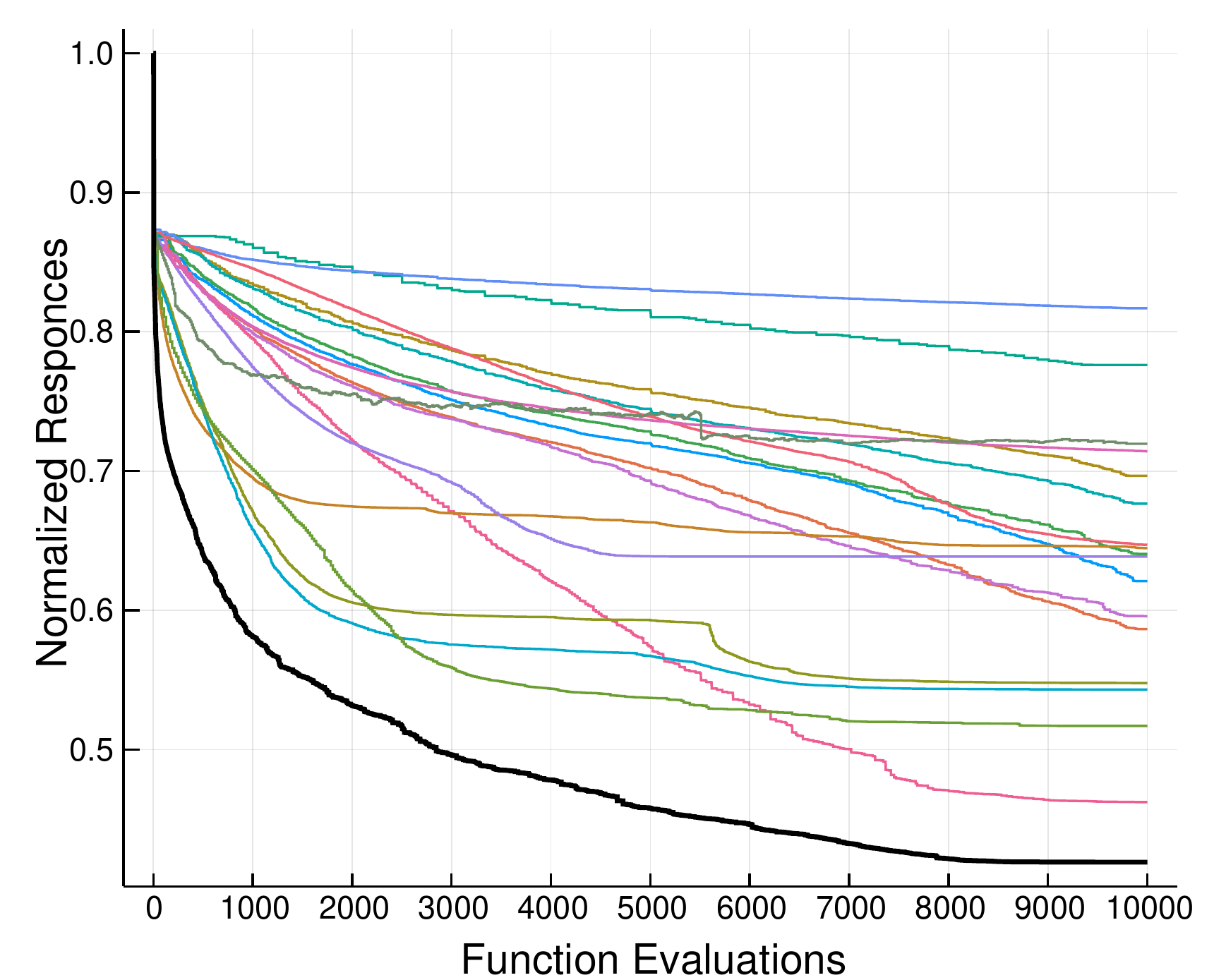}\label{fig:all_10times_20vars_10000evals}}
\subfloat[Utilized Optimizers]{\includegraphics[width=0.5\textwidth, keepaspectratio]{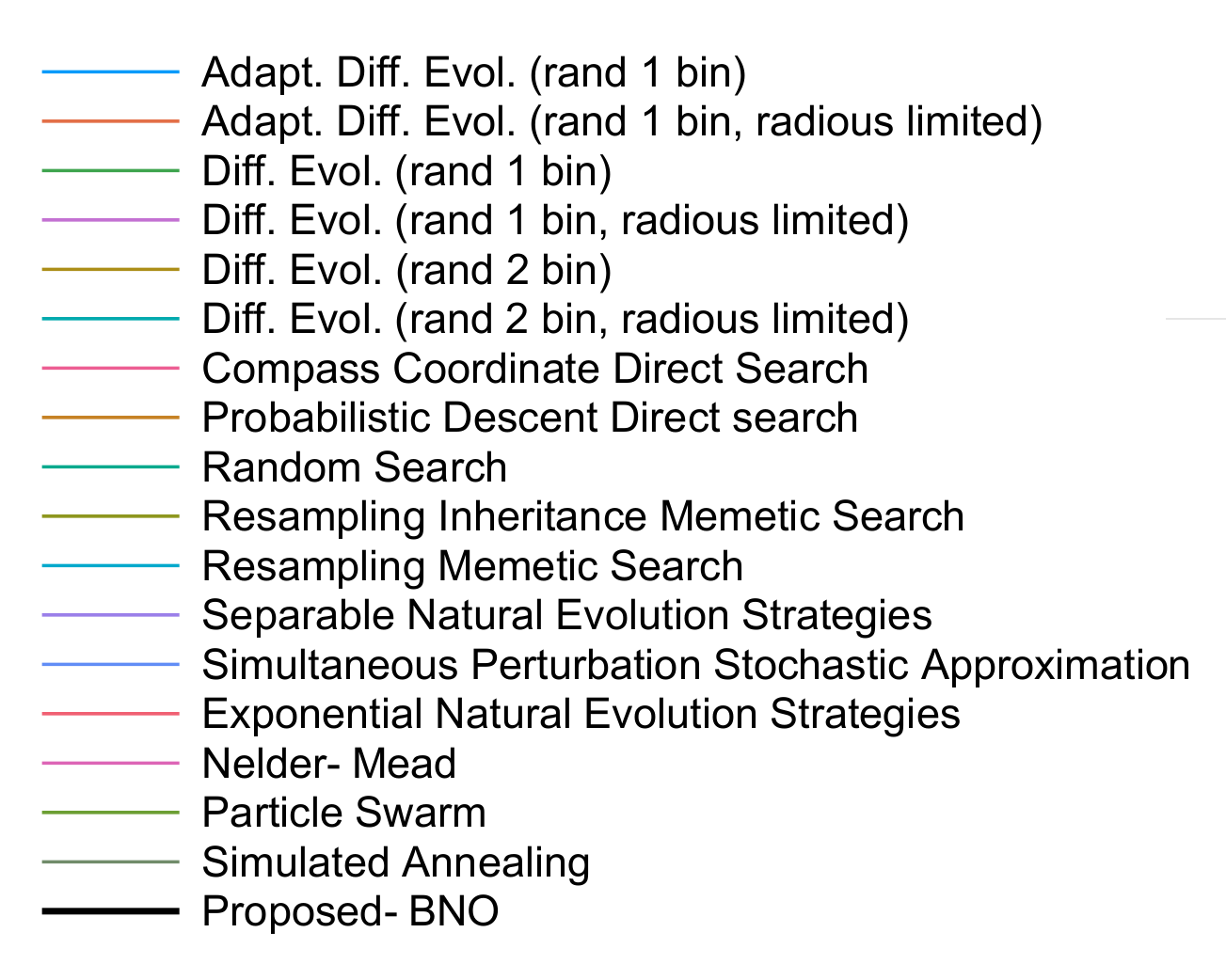}\label{fig:all_10times_50vars_1000evals}}
\caption[--]{Average Optimization history $h$ (Equation \ref{eq:h}, for all 13 Functions, repeated 10 times.}
\label{fig:run_tests}
\end{figure*}

\begin{table*}[!t]
\begin{center}
\captionsetup{justification=centering}
\caption {Average Minimum Values of $h$ Attained by each Optimizer} 
\begin{tabular}{@{}llll@{}}
\toprule
& $n=10$      & $n=20$      & $n=20$      \\ 
& $f_e=5000$ & $f_e=5000$ & $f_e=10000$ \\ \cmidrule{1-4}\morecmidrules\cmidrule{1-4}
\begin{tabular}[c]{@{}l@{}}Adapt. Diff. Evol. (rand 1 bin)\end{tabular}                    & 0.00991   & 0.02210   & 0.00852   \\ \midrule
\begin{tabular}[c]{@{}l@{}}Adapt. Diff. Evol. (rand 1 bin, radious limited)\end{tabular}   & 0.00734   & 0.01762   & 0.00482   \\ \midrule
\begin{tabular}[c]{@{}l@{}}Diff. Evol. (rand 1 bin)\end{tabular}                           & 0.01348   & 0.02196   & 0.01161   \\ \midrule
\begin{tabular}[c]{@{}l@{}}Diff. Evol. (rand 1 bin, radious limited)\end{tabular}          & 0.00760   & 0.01377   & 0.00564   \\ \midrule
\begin{tabular}[c]{@{}l@{}}Diff. Evol. (rand 2 bin)\end{tabular}                           & 0.02140   & 0.03712   & 0.02687   \\ \midrule
\begin{tabular}[c]{@{}l@{}}Diff. Evol. (rand 2 bin, radious limited)\end{tabular}          & 0.02072   & 0.03110   & 0.02008   \\ \midrule
\begin{tabular}[c]{@{}l@{}}Compass Coordinate Direct Search\end{tabular}                   & 0.00084   & 0.00118   & 0.00045   \\ \midrule
\begin{tabular}[c]{@{}l@{}}Probabilistic Descent Direct search\end{tabular}                & 0.00984   & 0.01424   & 0.01239   \\ \midrule
Random Search & 0.05678   & 0.08307   & 0.07936   \\ \midrule
\begin{tabular}[c]{@{}l@{}}Resampling Inheritance Memetic Search\end{tabular}              & 0.00238   & 0.00513   & 0.00243   \\ \midrule
\begin{tabular}[c]{@{}l@{}}Resampling Memetic Search\end{tabular}                          & 0.00129   & 0.00236   & 0.00223   \\ \midrule
\begin{tabular}[c]{@{}l@{}}Separable Natural Evolution Strategies\end{tabular}             & 0.02038   & 0.01165   & 0.01127   \\ \midrule
\begin{tabular}[c]{@{}l@{}}Simultaneous Perturbation Stochastic Approximation\end{tabular} & 0.13967   & 0.13998   & 0.13227   \\ \midrule
\begin{tabular}[c]{@{}l@{}}Exponential Natural Evolution Strategies\end{tabular}           & 0.01972   & 0.01415   & 0.01286   \\ \midrule
Nelder-Mead & 0.01858   & 0.03307   & 0.03452   \\ \midrule
Particle Swarm & 0.00254   & 0.00166   & 0.00136   \\ \midrule
\begin{tabular}[c]{@{}l@{}}Simulated Annealing\end{tabular}  & 0.03825   & 0.03812   & 0.03721   \\ \midrule
Proposed-ITSO  & 0.00001   & 0.00020   & 0.00017   \\ \bottomrule
\end{tabular}
\label{tab:all-opti}
\end{center}
\end{table*}

%\subsection{Structural Optimization}
%\label{sec:stru}

%\subsection{Reinforcement Learning}
%\label{sec:reinf}

\nomenclature{$n$}{number  of  dimensions  of  the  set $A$}
\nomenclature{$A$}{search space $\in {{\mathbb{R}}^{n}}$}
\nomenclature{${{f}_{e}}$}{the maximum function evaluations}
\nomenclature{$i=\left\{ 1,2,\ldots ,{{f}_{e}} \right\}$}{iterator for the $f_e$}
\nomenclature{$j=\left\{ 1,2,\ldots ,n \right\}$}{iterator for the dimensions of $A$}
\nomenclature{$\mathbf x^*$}{the argument of $f$ corresponding to the current optimum during the optimization process}
\nomenclature{$\mathbf x_m$}{the sought ${\operatorname {arg\,min} }\,f$}
\nomenclature{$h$}{convergence history of optimization algorithms}
%\begin{multicols}{2}
\printnomenclature[1cm]
%\end{multicols}

\section{Discussion and Conclusions}
\label{sec:conclusions}

In this work a novel approach was presented for the well known problem of finding the argument that minimizes a \textit{black-box}, function or system. A vast volume of approximation algorithms have been proposed, mainly heuristic, such as genetic, evolutionary, particle swarm, as well as their variations. However, they stem from \textit{nature-inspired} procedures, and hence their converge is investigated a-posteriori. Despite their efficiency, they are often deprecated by researchers, due to the lack of rigorous mathematical formulation, as well as complexity of implementation. To the contrary, the proposed algorithm, initiates its formulation from well established probabilistic definitions and theorems, and its implementation demands a few lines of computer code. Furthermore, the convergence properties were found stable, as a proof that the suggested framework attains the best possible solution in the fewest possible iterations. The numerical examples validate the theoretical results and may be reproduced by the provided computer code. We consider the suggested method as a powerful framework which may easily be adopted to the sought solution of any problem involving the minimization of a \textit{black-box} function.

\begin{appendices}
\section{Programming Code} 
The corresponding computer code is available on GitHub \url{https://github.com/nbakas/ITSO.jl}. The examples of Figure \ref{fig:run_tests} may be reproduced by running \textit{\_\_run.jl}. The sort version of the Algorithm \ref{al:ITSO} is available in Julia \cite{bezanson2017julia} (file \textit{ITSO-short.jl}), Octave \cite{Octave} (\textit{ITSOshort.m}), and Python \cite{python} (\textit{ITSO-short.py}). The implementation of the framework is integrated in a few lines of computer code, which can be easily adapted for case specific applications with high efficiency.

\section{\textit{Black-Box} Functions}
The following functions were used for the numerical experiments. Equations \ref{elliptic}, \ref{cigar} (Elliptic, Cigar), were utilized from \cite{liang2013problem}, Cigtab (Eq. \ref{cigtab}), Griewank \ref{griewank} from \cite{au2012eigenspace}, Quartic (Eq. \ref{quartic} from \cite{jamil2013literature}, Schwefel (Eq. \ref{schwefel}), Rastrigin (Eq. \ref{rastrigin}), Sphere (Eq. \ref{sphere}), and Ellipsoid (Eq. \ref{ellipsoid}) from \cite{finck2010real,BlackBoxOptim}, and Alpine (Eq. \ref{alpine}) from \cite{hussain2017common}. Equations \ref{x_j}, \ref{x-5}, \ref{sin-x}, were developed by the authors. The code implementation for the selected equations appears in file \textit{functions\_opti.jl} in the supplementary computer code. 

The exact variation used in this work is as follows, where we have adopted the notation presented in the Nomenclature section, where $i$ denotes the step of the optimization history, and $j$ the dimension of the design variable $x_{ij}$.  

\begin{equation}
    \label{elliptic}
    \begin{aligned}
        f_{elliptic}(\mathbf{x}_i)=\sum_{j=1}^n{c_j ( x_{ij} +\frac{3}{2})^2}, \text{where} \\
        \mathbf c=10^3 \{0,\frac{1}{n-1},\dots,1\}.
    \end{aligned}
\end{equation}
\begin{equation}
    \label{cigar}
    f_{cigar}(\mathbf{x}_i)=x_1^2+\sum_{j=2}^n{| x_{ij} |}.
\end{equation}
\begin{equation}
    \label{cigtab}
    f_{cigtab}(\mathbf{x}_i)=x_1^2+\sum_{j=2}^{n-1}{| x_{ij} |}+x_n^2.
\end{equation}
\begin{equation}
    \label{griewank}
    f_{griewank}(\mathbf{x}_i)=1+{\frac  {1}{4000}}\sum _{{j=1}}^{n}x_{ij}^{2}-\prod _{{j=1}}^{n}\cos \left({\frac {x_{ij}}{{\sqrt  {j}}}}\right).
\end{equation}
\begin{equation}
    \label{quartic}
    f_{quartic}(\mathbf{x}_i)=\sum_{j=1}^n{j (x_{ij}-2)^4}.
\end{equation}
\begin{equation}
    \begin{aligned}
    \label{schwefel}
    f_{schwefel}(\mathbf{x}_i)=\sum_{j=1}^n{ c_j^2}, \text{where} \\
    {c_j}=\sum_{k=1}^j{(x_{ik}-9)}.
    \end{aligned}
\end{equation}
\begin{equation}
    \label{rastrigin}
    \begin{aligned}
    f_{rastrigin}(\mathbf{x}_i)=10 n + \sum_{j=1}^n{(x_{ij}+\frac{7}{10})^2}
    \\
    -10 \sum_{j=1}^n{\cos (2 \pi (x_{ij}+\frac{7}{10})^2 )}.
    \end{aligned}
\end{equation}
\begin{equation}
    \label{sphere}
    f_{sphere}(\mathbf{x}_i)=\sum_{j=1}^n{(x_{ij}-\frac{13}{10})^2}.
\end{equation}
\begin{equation}
    \begin{aligned}
    \label{ellipsoid}
    f_{elipsoid}(\mathbf{x}_i)=\sum_{j=1}^n{(x_{ij}-\sqrt{2})^2}.
    \end{aligned}
\end{equation}
\begin{equation}
    \label{alpine}
    f_{alpine}(\mathbf{x}_i)=\sum_{j=1}^n{|x_{ij} \sin{x_{ij}} + \frac{1}{10} x_{ij}|}.
\end{equation}
\begin{equation}
    \label{x_j}
    f_{x\_j}(\mathbf{x}_i)=\sum_{j=1}^n{(x_{ij}-j-\frac{21}{10})^2}.
\end{equation}
\begin{equation}
    \label{x-5}
    f_{x\_5}(\mathbf{x}_i)=\sum_{j=1}^n{(x_{ij}-5)^2} - 5.
\end{equation}
\begin{equation}
    \label{sin-x}
    f_{sin\_x}(\mathbf{x}_i)=\sum_{j=1}^n{( \sin(x_{ij}+\frac{7}{10}) + \frac{(x_{ij}+\frac{7}{10})^2}{100}} ).
\end{equation}
\end{appendices}

\bibliographystyle{IEEEtran}
\bibliography{references}

\end{document}